\documentclass[11pt,a4paper]{article}

\usepackage[margin=1.2in]{geometry}

\usepackage[utf8]{inputenc}
\usepackage{amsmath}
\usepackage{mathtools}
\usepackage{amsfonts}
\usepackage{amssymb}
\usepackage[]{algorithm}
\usepackage{algorithmic}
\usepackage{amsthm}
\usepackage{hyperref}
\usepackage[noabbrev,capitalize]{cleveref}
\usepackage[english]{babel}
\usepackage{comment}
\usepackage{array,multirow}
\usepackage{graphicx}
\graphicspath{ {./images/} }
\usepackage{subcaption}
\usepackage[section]{placeins}
\usepackage{appendix}
\usepackage{color}
\usepackage[normalem]{ulem}

\newtheorem{theorem} {Theorem}
\newtheorem{lemma} {Lemma}
\newtheorem{definition} {Definition}
\newtheorem{corollary} {Corollary}

\newtheorem{assumption} {Assumption}

\newtheorem{remark} {Remark}

\def\u{{\mathbf{u}}}
\def\v{{\mathbf{v}}}
\def\z{{\mathbf{z}}}
\def\w{{\mathbf{w}}}
\def\s{{\mathbf{s}}}
\def\y{{\mathbf{y}}}

\def\b{{\mathbf{b}}}
\def\X{{\mathbf{X}}}
\def\Y{{\mathbf{Y}}}

\def\A{{\mathbf{A}}}

\def\M{{\mathbf{M}}}
\def\N{{\mathbf{N}}}
\def\I{{\mathbf{I}}}

\def\S{{\mathbf{S}}}
\def\Z{{\mathbf{Z}}}
\def\W{{\mathbf{W}}}
\def\U{{\mathbf{U}}}

\def\P{{\mathbf{P}}}

\def\G{{\mathbf{G}}}

\def\Sn{{\mathcal{S}_n}}

\DeclareMathOperator*{\argmin}{arg\,min}
\DeclareMathOperator*{\argmax}{arg\,max}

\newcommand{\mA}{\mathcal{A}}

\newcommand{\mK}{\mathcal{K}}
\newcommand{\mS}{\mathcal{S}}

\newcommand{\mbS}{\mathbb{S}}

\newcommand{\trace}{\textrm{Tr}}
\newcommand{\rank}{\textrm{rank}}

\newcommand{\diag}{\textrm{diag}}
\newcommand{\support}{\textrm{support}}
\newcommand{\reals}{\mathbb{R}}
\newcommand{\sign}{\textrm{sign}}

\title{Low-Rank Extragradient Method for Nonsmooth and Low-Rank Matrix Optimization Problems\footnote{This version corrects an error in the original paper published in NeurIPS 2021 \cite{GarberKaplan21}: while the version \cite{GarberKaplan21} provides convergence rates w.r.t. the best iterate (which under the assumptions of the paper is guaranteed to be low-rank), this corrected version provides the same rates but for the ergodic sequence, i.e., the averaged iterate (which, under our assumptions, is the average of low-rank iterates).}}
\author{Dan Garber \\ {\small Technion - Israel Institute of Technology}\\ {\small \texttt{dangar@technion.ac.il}}
\and
Atara Kaplan \\  {\small Technion - Israel Institute of Technology} \\ {\small  \texttt{ataragold@campus.technion.ac.il}}
}

\date{}

\begin{document} 

\maketitle

\begin{abstract}
Low-rank and nonsmooth matrix optimization problems capture many fundamental tasks in statistics and machine learning.
While significant progress has been made in recent years in developing efficient methods for \textit{smooth} low-rank optimization problems that avoid maintaining high-rank matrices and computing expensive high-rank SVDs, advances for nonsmooth problems have been slow paced. 

In this paper we consider standard convex relaxations for such problems. Mainly, we prove that under a natural \textit{generalized strict complementarity} condition and under the relatively mild assumption that the nonsmooth objective can be written as a maximum of smooth functions, the \textit{extragradient method}, when initialized with a ``warm-start'' point, converges to an optimal solution with rate $O(1/t)$ while requiring only two \textit{low-rank} SVDs per iteration. We give a precise trade-off between the rank of the SVDs required and the radius of the ball in which we need to initialize the method.
We support our theoretical results with empirical experiments on several nonsmooth low-rank matrix recovery tasks, demonstrating that using simple initializations, the extragradient method produces exactly the same iterates when full-rank SVDs are replaced with SVDs of rank that matches the rank of the (low-rank) ground-truth matrix to be recovered.
\end{abstract}

\section{Introduction}

Low-rank and nonsmooth matrix optimization problems have many important applications in statistics, machine learning, and related fields, such as \textit{sparse PCA} \cite{sparsePCA1, sparsePCA2}, \textit{robust PCA} \cite{robustPCA1, robustPCA2, robustPCA3, robustPCA4, robustPCA5},  \textit{phase synchronization} \cite{phaseSyncronization1, phaseSyncronization2, phaseSyncronization3}, \textit{community detection and stochastic block models} \cite{abbe2017community}\footnote{in \cite{phaseSyncronization1, phaseSyncronization2, phaseSyncronization3} and \cite{abbe2017community} the authors consider SDPs with linear objective function and affine constraints of the form $\mA(\X)=\b$. By incorporating the linear constraints into the objective function via a $\ell_2$ penalty term of the form $\lambda\Vert{\mA(\X)-\b}\Vert_2$, $\lambda > 0$, we obtain a nonsmooth objective function.},  \textit{low-rank and sparse covariance matrix recovery} \cite{lowRankAndSparse}, \textit{robust matrix completion} \cite{klopp2017robust, chen2011robust}, and more. For many of these problems, convex relaxations, in which one replaces the nonconvex low-rank constraint with a trace-norm constraint, have been demonstrated in numerous papers to be highly effective both in theory (under suitable assumptions) and empirically (see references above).
These convex relaxations can be formulated as the following general nonsmooth optimization problem:
\begin{align} \label{nonSmoothProblem}
\min_{\X\in\Sn} g(\X),
\end{align} 
where $g:\mathbb{S}^n\rightarrow\reals$ is convex but nonsmooth, and $\mathcal{S}_n=\{\X\in\mathbb{S}^n\ \vert\ \trace(\X)=1,\ \X\succeq0\}$ is the spectrahedron in $\mathbb{S}^n$, $\mbS^n$ being the space of $n\times n$ real symmetric matrices.

Problem \eqref{nonSmoothProblem}, despite being convex, is notoriously difficult to solve in large scale. The simplest and most general approach applicable to it is the \textit{projected subgradient method}  \cite{beckOptimizationBook, sebastienMD}, which requires on each iteration to compute a Euclidean projection onto the spectrahedron $\mS_n$, which in worst case amounts to $O(n^3)$ runtime per iteration. In many applications $g(\X)$ follows a composite model, i.e., $g(\X) = h(\X) + w(\X)$, where $h(\cdot)$ is convex and smooth and $w(\cdot)$ is convex and nonsmooth but admits a simple structure (e.g., nonsmooth regularizer). For such composite objectives, without the spectrahedron constraint, proximal methods such as FISTA \cite{FISTA} or splitting methods such as ADMM  \cite{ADMM} are often very effective. However, with the spectrahderon constraint, all such methods require on each iteration to apply a subprocedure (e.g., computing the proximal mapping) which in worst case amounts to at least $O(n^3)$ runtime. 
A third type of off-the-shelf methods include those which are based on the \textit{conditional gradient method} and adapted to nonsmooth problems, see for instance \cite{odor2016frank, lowRankAndSparseOurs, NEURIPS2020_8f468c87, locatello2019stochastic}. The advantage of such methods is that  no expensive high-rank SVD computations are needed. Instead, only a single leading eigenvector computation (i.e., a rank-one SVD) per iteration is required. However, similarly to the subgradient method, these suffer from slow convergence rates ($O(1/\sqrt{t})$ at best, when $t$ is the iteration counter) even when the nonsmooth problem admits favorable structure (as we detail below). Thus, to conclude, standard first-order methods for Problem \eqref{nonSmoothProblem} require in worst case $\Omega(n^3)$ runtime per iteration or suffer from worst-case slow convergence rates.

In the recent works \cite{garberLowRankSmooth, garberStochasticLowRank} it was established that for smooth objective functions, the high-rank SVD computations required for Euclidean projections onto the spectrahedron in standard gradient methods, can be replaced with low-rank SVDs in the  close proximity of  a low-rank optimal solution. This is significant since the runtime to compute a rank-$r$ SVD of a given $n\times n$ matrix using efficient iterative methods typically scales with $rn^2$ (and further improves when the matrix is sparse), instead of $n^3$ for a full-rank SVD.
These results depend on the existence of eigen-gaps in the gradient of the optimal solution, which we refer to as a \textit{generalized strict complementarity condition}.  
These results also hinge on a unique property of the Euclidean projection onto the spectrahedron. The projection onto the spectrahedron of a matrix $\X\in\mathbb{S}^n$, which admits an eigen-decomposition $\X=\sum_{i=1}^n\lambda_i\v_i\v_i^{\top}$, is given by 
\begin{align}\label{eq:euclidProj}
\Pi_{\mathcal{S}_n}[\X]=\sum_{i=1}^n\max\{0,\lambda_i-\lambda\}\v_i\v_i^{\top},
\end{align}
where $\lambda\in\reals$ is the unique scalar satisfying  $\sum_{i=1}^n\max\{0,\lambda_i-\lambda\}=1$. This operation thus truncates all eigenvalues that are smaller than $\lambda$, while leaving the eigenvectors unchanged, thereby returning a matrix with rank equal to the number of eigenvalues greater than $\lambda$. Importantly, when the projection of $\X$ onto $\mS_n$ is of rank $r$, only the first $r$ components in the eigen-decomposition of $\X$ are required to compute it in the first place, and thus, only a rank-$r$ SVD of $\X$ is required.
In other words and simplifying, \cite{garberLowRankSmooth, garberStochasticLowRank}  show that under strict complementary, at the proximity of an optimal solution of rank $r$, the exact Euclidean  projection  equals the rank-$r$ truncated projection given by:
\begin{align} \label{truncatedProjection}
\widehat{\Pi}_{\mathcal{S}_n}^r[\X]:=\Pi_{\mathcal{S}_n}\left[\sum_{i=1}^r\lambda_i\v_i\v_i^{\top}\right].
\end{align}

Extending the results of \cite{garberLowRankSmooth, garberStochasticLowRank} to the nonsmooth setting is difficult since the smoothness assumption is critical to the analysis. 
Moreover, while  \cite{garberLowRankSmooth, garberStochasticLowRank} rely on certain eigen-gaps in the gradients at optimal points, for nonsmooth problems, since the subdifferential set is often not a singleton, it is not likely that a similar eigen-gap property holds for all subgradients of an optimal solution.

In this paper we show that under the mild assumption that Problem \eqref{nonSmoothProblem} can be formulated as a smooth convex-concave saddle-point problem, i.e., the nonsmooth term can be written as a maximum over (possibly infinite number of) smooth convex functions, we can obtain results in the spirit of \cite{garberLowRankSmooth, garberStochasticLowRank}. Concretely, we show that if a generalized strict complementarity (GSC) assumption holds for a low-rank optimal solution (see Assumption \ref{ass:strictcompNonSmooth} in the sequel), the \textit{extragradient method} for smooth convex-concave saddle-point problems \cite{extragradientK,Nemirovski} (see Algorithm \ref{alg:EG} below), when initialized in the proximity of  the optimal solution, converges with its original convergence rate of $O(1/t)$, while requiring only two low-rank SVDs per iteration\footnote{note that the extradgradient method computes two projected-gradient steps on each iteration, and thus two SVDs are needed per iteration.}. It is important to recall that while the extragradient method requires two SVDs per iteration, it has the benefit of a fast $O(1/t)$ convergence rate, while simpler saddle-point methods such as mirror-descent-based only achieve a $O(1/\sqrt{t})$ rate \cite{sebastienMD}.

Our contributions can be summarized  as follows:
\begin{itemize}

\item
We prove that even under (standard) strict complementarity, the projected subgradient method, when initialized with a ``warm-start", may produce iterates with rank higher than that of the optimal solution. This phenomena further motivates our saddle-point approach. See Lemma \ref{lemma:negativeExampleNonSmooth}.

\item
We suggest a generalized strict complementarity (GSC) condition for saddle-point problems and prove that when $g(\cdot)$ --- the objective function in Problem \eqref{nonSmoothProblem}, admits a highly popular saddle-point structure (one which captures all applications we mentioned in this paper),  GSC w.r.t. an optimal solution to  Problem  \eqref{nonSmoothProblem} implies GSC (with the same parameters) w.r.t. a corresponding optimal solution of the equivalent saddle-point problem (the other direction always holds). See Section \ref{sec:smooth2Saddle}.

\item
\textbf{Main result:} we prove that for a smooth convex-concave saddle-point problem and an optimal solution which satisfies GSC, the extragradient method, when initialized with a ``warm-start", converges with its original rate of $O(1/t)$ while requiring only two low-rank SVDs per iteration. Moreover, we prove GSC facilitates a precise and powerful tradeoff: increasing the rank of SVD computations (beyond the rank of the optimal solution) can significantly increase the radius of the ball in which the method needs to be initialized. See Theorem \ref{theroem:allPutTogether}.

\item We present extensive numerical evidence that demonstrate both the plausibility of the GSC assumption in various tasks, and more importantly, demonstrate that indeed the extragradient method with simple initialization converges correctly (i.e., produces exactly the same sequences of iterates) when the rank of the SVDs used to compute the (truncated) projections matches the rank of the (low-rank) ground-truth matrix to be recovered, instead of naively using full-rank SVDs (as suggested by \eqref{eq:euclidProj}). See Section \ref{sec:expr}.
%$\rank(\X^*)$-SVDs are used to compute the projections (as in \eqref{truncatedProjection}), instead of  full-rank SVDs, where $\X^*$ is an optimal solution. See Section \ref{sec:expr}.
\end{itemize}

\subsection{Additional related work}
Since, as in the works  \cite{garberLowRankSmooth, garberStochasticLowRank} mentioned before which deal with smooth objectives,  strict complementarity plays a key role in our analysis, we refer the interested reader to the recent works \cite{garber2019linear, spectralFrankWolfe, ding2020kfw, garber2020efficient} which also exploit this property for efficient smooth and convex optimization over the spectrahedron. Strict complementarity has also played an instrumental role in two recent and very influential works which used it to prove linear convergence rates for proximal gradient methods \cite{zhou2017unified, drusvyatskiy2018error}.

Besides convex relaxations such as Problem \eqref{nonSmoothProblem}, considerable advances have been made in the past several yeas in developing efficient \textit{nonconvex} methods with global convergence guarantees for low-rank matrix problems. 
In \cite{nonConvexFactorizedAssumptions} the authors consider semidefinite programs and prove that under a smooth manifold assumption on the constraints, such methods converge to the optimal global solution.
In \cite{nonconvexFactorizedGlobal} the authors prove global convergence of factorized nonconvex gradient descent from a ``warm-start'' initialization point for non-linear smooth minimization on the positive semidefinite cone. Very recently, \cite{charisopoulos2021low} has established, under statistical conditions, fast convergence results from ``warm-start'' initialization of nonconvex first-order methods, when applied to nonsmooth nonconvex matrix recovery problems which are based on the explicit factorization of the low-rank matrix. A result of similar flavor concerning nonsmooth and nonconvex formulation of robust recovery of low-rank matrices from random linear measurements was presented in \cite{li2020nonconvex}.  
Finally,  several recent works have considered nonconvex low-rank regularizers which result in nonconvex nonsmooth optimization problems, but guarantee convergence only to a stationary point \cite{lu2014generalized, yao2018large}.

\section{Strict Complementarity for Nonsmooth Optimization and Difficulty of Applying Low-Rank Projected Subgradient Steps}
\label{sec:strictComp}

Our analysis of the nonsmooth Problem \eqref{nonSmoothProblem} naturally depends on  certain subgradients of an optimal solution which, in many aspects, behave like the gradients of smooth functions. The existence of such a subgradient is guaranteed from the first-order optimality condition for constrained convex minimization problems:
\begin{lemma}[first-order optimality condition, see \cite{beckOptimizationBook}] \label{optCondition}
Let $g:\mathbb{S}^n\rightarrow\reals$ be a convex function. Then $\X^*\in\Sn$ minimizes $g$ over $\Sn$ if and only if there exists a subgradient $\G^*\in\partial g(\X^*)$ such that $\langle \X-\X^*,\G^*\rangle\ge0$ for all $\X\in\Sn$.
\end{lemma}

For some $\G^*\in\partial g(\X^*)$ which satisfies the first-order optimality condition for an optimal solution $\X^*$, if the multiplicity of the smallest eigenvalue equals $r^*=\rank(\X^*)$, then it can be shown that the optimal solution satisfies a strict complementarity assumption. The equivalence between a standard strict complementarity assumption on some low-rank optimal solution of a \textit{smooth} optimization problem over the spectrahedron and an eigen-gap in the gradient of the optimal solution was established in \cite{spectralFrankWolfe}. We generalize this equivalence to also include nonsmooth problems. The proof follows similar arguments and is given in \cref{sec:proofKKT}.

\begin{definition}[strict complementarity] \label{def:strict}
An optimal solution $\X^*\in\Sn$ of rank $r^*$ for Problem \eqref{nonSmoothProblem} satisfies the strict complementarity assumption with parameter $\delta >0$, if there exists an optimal solution of the dual problem\footnote{Denote $q(\Z,s)=\min_{\X\in\mathbb{S}^n} \lbrace g(\X)+s(1-\trace(\X))-\langle\Z,\X\rangle\rbrace$.
The dual problem of Problem \eqref{nonSmoothProblem} can be written as:
$ \max_{\lbrace\Z\succeq0,\ s\in\reals\rbrace} \lbrace q(\Z,s) \ \vert \ (\Z,s)\in \textrm{dom}(q)\rbrace$.}
 $(\Z^*,s^*)\in\mbS^n\times\reals$ such that $\rank(\Z^*)=n-r^*$, and $\lambda_{n-r^*}(\Z^*)\ge\delta$. 
\end{definition}

\begin{lemma} \label{lemma:kkt}
Let $\X^*\in\Sn$ be a rank-$r^*$ optimal solution to Problem \eqref{nonSmoothProblem}. $\X^*$ satisfies the (standard) strict complementarity assumption with parameter $\delta>0$ if and only if there exists a subgradient $\G^*\in\partial g(\X^*)$ such that $\langle \X-\X^*,\G^*\rangle\ge0$ for all $\X\in\mathcal{S}_n$ and $\lambda_{n-r^*}(\G^*)-\lambda_{n}(\G^*)\ge\delta$.
\end{lemma}

Throughout this paper we assume a weaker and more general assumption than strict complementarity, namely generalized strict complementarity (GSC), which we present now.

\begin{assumption}[generalized strict complementarity]\label{ass:strictcompNonSmooth}
We say an optimal solution $\X^*$ to Problem \eqref{nonSmoothProblem} satisfies the generalized strict complementarity assumption with parameters $r,\delta$, if there exists a subgradient $\G^*\in\partial g(\X^*)$ such that $\langle \X-\X^*,\G^*\rangle\ge0$ for all $\X\in\mathcal{S}_n$ and $\lambda_{n-r}(\G^*)-\lambda_{n}(\G^*) \geq \delta$.
\end{assumption}

In \cite{garberLowRankSmooth} the author presents several characteristic properties of the gradient of the  optimal solution in optimization problems over the spectrahedron. 
Using the existence of  subgradients which satisfy the condition in \cref{optCondition}, we can extend these properties also to the nonsmooth setting.
The following lemma shows that GSC with parameters $(r,\delta)$ for some $\delta>0$ (\cref{ass:strictcompNonSmooth}) is a sufficient condition for the optimal solution to be of rank at most $r$. The proof follows immediately from the proof of the analogous Lemma 7 in \cite{garberLowRankSmooth}, by replacing the gradient of the optimal solution with a subgradient for which the first-order optimality condition holds. 
\begin{lemma} \label{lemma:Xopt-subgradEigen}
Let $\X^*$ be an optimal solution to Problem \eqref{nonSmoothProblem} and write its eigen-decomposition as $\X^*=\sum_{i=1}^{r^*}{\lambda_i\v_i\v_i^T}$. Then, any subgradient $\G^*\in\partial g(\X^*)$ which satisfies $\langle \X-\X^*,\G^*\rangle\ge0$ for all $\X\in\mathcal{S}_n$, admits an eigen-decomposition such that the set of vectors $\{\v_i\}_{i=1}^{r^*}$ is a set of leading eigenvectors of $(-\G^*)$ which corresponds to the eigenvalue $\lambda_1(-\G^*)=-\lambda_n(\G^*)$. Furthermore, there exists at least one such subgradient. 
\end{lemma}

One motivation for assuming (standard) strict complementarity (\cref{ass:strictcompNonSmooth} with parameters $r=\rank(\X^*)$ and $\delta>0$) is that it guarantees a certain notion of robustness of the problem to small perturbations in the parameters. It is well known (see for instance \cite{beckOptimizationBook}) that a projected subgradient step from $\X^*$ with respect to a subgradient $\G^*\in\partial g(\X^*)$ for which the first-order optimality condition holds, returns the optimal solution $\X^*$ itself. This implies that $\rank\left( \Pi_{\Sn} [ \X^* - \eta\G^* ]\right)=\rank(\X^*)$ (here $\eta$ is the step-size).
Without (standard) strict complementarity however, a small change in the parameters could result in a higher rank matrix. This is captured in the following lemma which is analogous to Lemma 3 in \cite{garberStochasticLowRank}, where again the proof is straightforward from the proof in  \cite{garberStochasticLowRank} by replacing the gradient of the optimal solution with a subgradient for which the first-order optimality condition holds. 
\begin{lemma}
Let $\X^*$ be an optimal solution of rank $r^*$ to Problem \eqref{nonSmoothProblem}. Let $\G^*\in\partial g(\X^*)$ be a subgradient at $\X^*$ such that $\langle \X-\X^*,\G^*\rangle\ge0$ for all $\X\in\mathcal{S}_n$. Then, $\lambda_{n-r^*}(\G^*)=\lambda_{n}(\G^*)$ if and only if for any arbitrarily small $\zeta>0$ it holds that 
$\rank\left( \Pi_{(1+\zeta)\mathcal{S}_n} [ \X^* - \eta\G^* ]\right)>r^*$,
where $\eta>0$, $(1+\zeta)\mathcal{S}_n = \{(1+\zeta)\X \ \vert\ \X\in\mathcal{S}_n\}$, and $\Pi_{(1+\zeta)\mathcal{S}_n}{[\cdot]}$ denotes the Euclidean projection onto the set $(1+\zeta)\mathcal{S}_n$.
\end{lemma}

\subsection{The challenge of applying low-rank projected subgradient steps}

We now demonstrate the difficulty of replacing the full-rank SVD computations required in projected subgradient steps over the spectrahedron, with their low-rank SVD counterparts when attempting to solve Problem \eqref{nonSmoothProblem}. We prove that a projected subgradient step from a point arbitrarily close to a low-rank optimal solution --- even one that satisfies strict complementarity (Definition \ref{def:strict}), may result in a higher rank matrix.  The problem on which we demonstrate this phenomena is a well known convex formulation of the \textit{sparse PCA} problem \cite{d2007direct}.

\begin{lemma}[failure of low-rank subgradient descent on sparse PCA] \label{lemma:negativeExampleNonSmooth}
Consider the problem
\begin{align*}
\min_{\X\in\Sn}\{g(\X):=-\left\langle\z\z^{\top}+\z_{\perp}\z_{\perp}^{\top},\X\right\rangle+\frac{1}{2k}\Vert\X\Vert_1\},
\end{align*}
where $\z=(1/\sqrt{k},\ldots,1/\sqrt{k},0,\ldots,0)^{\top}$ is supported on the first $k$ entries,  $\z_{\perp}=(0,\ldots,0,\\ 1/\sqrt{n-k},\ldots,1/\sqrt{n-k})^{\top}$ is supported on the last $n-k$ entries, and $k\le n/4$. Then, $\z\z^{\top}$ is a rank-one optimal solution for which  strict complementarity holds.
However, for any $\eta<\frac{2}{3}$ and any $\v\in\reals^n$ such that $\Vert\v\Vert=1$, $\support(\v)\subseteq\support(\z)$, and $\langle\z,\v\rangle^2=1 - \frac{1}{2}\Vert{\v\v^{\top}-\z\z^{\top}}\Vert_F^2 \ge1-\frac{1}{2k^2}$, it holds that
\begin{align*}
\rank\left(\Pi_{\Sn}[\v\v^{\top}-\eta\G_{\v\v^{\top}}]\right)>1,
\end{align*}
where $\G_{\v\v^{\top}}=-\z\z^{\top}-\z_{\perp}\z_{\perp}^{\top}+\frac{1}{2k}\sign(\v\v^{\top})\in\partial g(\v\v^{\top})$.
\end{lemma}
Note that the subgradient of the $\ell_1$-norm which we choose for the projected subgradient step simply corresponds to the sign function, which is arguably the most natural choice.

\begin{proof}

$\z\z^{\top}$ is a rank-one optimal solution for this problem since for the subgradient $k\z\z^{\top}+2k\z_{\perp}\z_{\perp}^{\top}\in\partial \left(\Vert\z\z^{\top}\Vert_1\right)$ the first-order optimality condition holds. Indeed, for all $\X\in\Sn$ 
\begin{align} \label{eq:ExampleFirstOrderOptCondHolds}
\langle\X-\z\z^{\top},-\z\z^{\top}-\z_{\perp}\z_{\perp}^{\top}+\frac{1}{2}\z\z^{\top}+\z_{\perp}\z_{\perp}^{\top}\rangle = \langle\X-\z\z^{\top},-\frac{1}{2}\z\z^{\top}\rangle\ge0.
\end{align}

For the subgradient $-\frac{1}{2}\z\z^{\top}\in\partial g(\z\z^{\top})$ there is a gap $\lambda_{n-1}(-\frac{1}{2}\z\z^{\top})-\lambda_{n}(-\frac{1}{2}\z\z^{\top})=\frac{1}{2}>0$, and as we showed in \eqref{eq:ExampleFirstOrderOptCondHolds} the first order optimality condition holds for $-\frac{1}{2}\z\z^{\top}$. Thus, by \cref{lemma:kkt} the optimal solution $\z\z^{\top}$ satisfies standard strict complementarity. 

We will show that the projection onto the spectrahedron of a subgradient step from $\v\v^{\top}$ with respect to the natural subgradient of the $\ell_1$-norm $\sign(\v\v^{\top})\in\partial \left(\Vert\v\v^{\top}\Vert_1\right)$ returns a rank-2 solution.

It holds that
\begin{align*}
1-\frac{1}{2k^2}\le\langle\z\z^{\top},\v\v^{\top}\rangle=\frac{1}{2}\left(\Vert\z\z^{\top}\Vert_F^2+\Vert\v\v^{\top}\Vert_F^2-\Vert\v\v^{\top}-\z\z^{\top}\Vert_F^2\right) = 1- \frac{1}{2}\Vert\v\v^{\top}-\z\z^{\top}\Vert_F^2,
\end{align*}
and equivalently
\begin{align*}
\sum_{i=1}^k\sum_{j=1}^k\left(\frac{1}{k}-(\v\v^{\top})_{ij}\right)^2=\Vert\v\v^{\top}-\z\z^{\top}\Vert_F^2\le\frac{1}{k^2}.
\end{align*}
Therefore, for every $i,j\in\lbrace 1,\ldots,k\rbrace$ it holds that
\begin{align*}
\left\vert(\v\v^{\top})_{ij}-\frac{1}{k}\right\vert\le\frac{1}{k},
\end{align*}
which implies that $0\le(\v\v^{\top})_{ij}\le\frac{2}{k}$. Therefore, $k\z\z^{\top}=\sign(\v\v^{\top})\in\partial \left(\Vert\v\v^{\top}\Vert_1\right)$.

Taking a projected subgradient step from $\v\v^{\top}$ with respect to the subgradient $-\z\z^{\top}-\z_{\perp}\z_{\perp}^{\top}+\frac{1}{2}\z\z^{\top}\in\partial g(\v\v^{\top})$ has the form
\begin{align*}
&\Pi_{\Sn}\left[\v\v^{\top}-\eta\left(-\z\z^{\top}-\z_{\perp}\z_{\perp}^{\top}+\frac{1}{2}\z\z^{\top}\right)\right]
 = \Pi_{\Sn}\left[\v\v^{\top}+\frac{\eta}{2}\z\z^{\top}+\eta\z_{\perp}\z_{\perp}^{\top}\right].
% \\ &  = \Pi_{\Sn}\left[\lambda_1\v_1\v_1^{\top}+\lambda_2\v_2\v_2^{\top}+\eta\z_{\perp}\z_{\perp}^{\top}\right].
\end{align*}

Since $\support(\v)\subseteq\support(\z)$ it holds that $\left(\v\v^{\top}+\frac{\eta}{2}\z\z^{\top}\right)\perp\z_{\perp}\z_{\perp}^{\top}$.
$\v\v^{\top}+\frac{\eta}{2}\z\z^{\top}$ is a rank-2 matrix and so we can denote the eigen-decomposition of $\v\v^{\top}+\frac{\eta}{2}\z\z^{\top}+\eta\z_{\perp}\z_{\perp}^{\top}$ as $\v\v^{\top}+\frac{\eta}{2}\z\z^{\top}+\eta\z_{\perp}\z_{\perp}^{\top}=\lambda_1\v_1\v_1^{\top}+\lambda_2\v_2\v_2^{\top}+\eta\z_{\perp}\z_{\perp}^{\top}$, where $\lambda_1\ge\lambda_2$. 
Thus, invoking \eqref{eq:euclidProj} to calculate the projection we need to find the scalar $\lambda\in\reals$ for which the following holds.
\begin{align*}
\max\left\lbrace\lambda_1-\lambda,0\right\rbrace+\max\left\lbrace\lambda_2-\lambda,0\right\rbrace+\max\lbrace\eta-\lambda,0\rbrace
+\sum_{i=4}^n\max \lbrace0-\lambda,0\rbrace = 1.
\end{align*}

$\lambda_1$ is the largest eigenvalue of $\v\v^{\top}+\frac{\eta}{2}\z\z^{\top}+\eta\z_{\perp}\z_{\perp}^{\top}$ since, under our assumption that $\eta < 2/3$, we have that
\[\lambda_1\ge \frac{1}{2}(\lambda_1+\lambda_2)=\frac{1}{2}\trace(\v\v^{\top}+\frac{\eta}{2}\z\z^{\top})=\frac{1}{2}+\frac{\eta}{4}>\eta.\]
Therefore, $\lambda<\lambda_1\le \lambda_1+\lambda_2=\trace(\v\v^{\top}+\frac{\eta}{2}\z\z^{\top})=1+\frac{\eta}{2}$.

In addition, $\max\lbrace\lambda_2,\eta\rbrace\ge\eta$. Therefore, 
\[\lambda_1-\max\lbrace\lambda_2,\eta\rbrace \le \lambda_1 + \lambda_2 - \eta = 1+\frac{\eta}{2}-\eta <1,\]
and so we must have that $\lambda<\max\lbrace\lambda_2,\eta\rbrace\le\lambda_1$. 

This implies that both $\max\left\lbrace\lambda_1-\lambda,0\right\rbrace>0$
and $\max\lbrace\max\lbrace\lambda_2,\eta\rbrace-\lambda,0\rbrace>0$. Thus, using  \eqref{eq:euclidProj} we conclude that 
\[\rank\left(\Pi_{\Sn}\left[\v\v^{\top}-\eta\left(-\z\z^{\top}-\z_{\perp}\z_{\perp}^{\top}+\frac{1}{2}\z\z^{\top}\right)\right]\right)\ge2.\]

\end{proof}

\section{From Nonsmooth to Saddle-Point Problems}\label{sec:smooth2Saddle}

To circumvent the difficulty demonstrated in \cref{lemma:negativeExampleNonSmooth} in incorporating low-rank SVDs into standard subgradient methods for solving Problem \eqref{nonSmoothProblem}, we propose tackling the nonsmooth problem with saddle-point methods.

We assume the nonsmooth Problem \eqref{nonSmoothProblem} can be written as a maximum of smooth functions, i.e., $g(\X)=\max_{\y\in\mathcal{K}}f(\X,\y)$, where $\mathcal{K}\subset\mathbb{Y}$ is some compact and convex subset of the finite linear space over the reals $\mathbb{Y}$ onto which it is efficient to compute Euclidean projections. We assume $f(\cdot,\y)$ is convex for all $\y\in\mathcal{K}$ and $f(\X,\cdot)$ is concave for all $\X\in\Sn$. That is, we rewrite Problem \eqref{nonSmoothProblem} as the following equivalent saddle-point  problem:
\begin{align} \label{problem1}
 \min_{\X\in\mS_n}\max_{\y\in\mathcal{K}}{f(\X,\y)}.
\end{align}

Finding an optimal solution to problem \eqref{problem1} is equivalent to finding a saddle-point $(\X^*,\y^*)\in{\Sn\times\mathcal{K}}$ such that for all $\X\in\Sn$ and $\y\in\mathcal{K}$, 
\begin{align*}
f(\X^*,\y) \le f(\X^*,\y^*) \le f(\X,\y^*).
\end{align*}

We make a standard assumption that $f(\cdot,\cdot)$ is smooth with respect to all the components. That is, we assume there exist $\beta_X,\beta_y,\beta_{Xy},\beta_{yX}\ge0$ such that for any $\X,\tilde{\X}\in\Sn$ and $\y,\tilde{\y}\in\mathcal{K}$ the following four inequalities hold:
\begin{align*}
& \Vert\nabla_{\X}f(\X,\y)-\nabla_{\X}f(\tilde{\X},\y)\Vert_{F} \le\beta_{X}\Vert\X-\tilde{\X}\Vert_{F}, \\
&\Vert\nabla_{\y}f(\X,\y)-\nabla_{\y}f(\X,\tilde{\y})\Vert_{2} \le\beta_{y}\Vert\y-\tilde{\y}\Vert_{2}, \\
& \Vert\nabla_{\X}f(\X,\y)-\nabla_{\X}f(\X,\tilde{\y})\Vert_{F} \le\beta_{Xy}\Vert\y-\tilde{\y}\Vert_{2}, \\
& \Vert\nabla_{\y}f(\X,\y)-\nabla_{\y}f(\tilde{\X},\y)\Vert_{2} \le\beta_{yX}\Vert\X-\tilde{\X}\Vert_{F},
\end{align*}
where $\nabla_{\X}f=\frac{\partial f}{\partial \X}$ and $\nabla_{\y}f=\frac{\partial f}{\partial \y}$.

We denote by $\beta$ the full Lipschitz parameter of the gradient, that is for any $\X,\tilde{\X}\in\Sn$ and $\y,\tilde{\y}\in\mathcal{K}$, 
\begin{align*}
\Vert(\nabla_{\X}f(\X,\y),-\nabla_{\y}f(\X,\y))-(\nabla_{\X}f(\tilde{\X},\tilde{\y}),-\nabla_{\y}f(\tilde{\X},\tilde{\y}))\Vert \le\beta\Vert(\X,\Y)-(\tilde{\X},\tilde{\y})\Vert,
\end{align*}
where $\Vert\cdot\Vert$ denotes the Euclidean norm over the product space $\Sn\times\mathbb{Y}$.

%The relationship between the full Lipschitz parameter $\beta$ and its components $\beta_X,\beta_y,\beta_{Xy},\beta_{yX}$ can be written as $\beta= \sqrt{2}\max\left\lbrace\sqrt{\beta_{X}^2+\beta_{yX}^2},\sqrt{\beta_{y}^2+\beta_{Xy}^2}\right\rbrace$. See proof in \cref{appedix:conectionBeta}.

To establish the relationship between $\beta$ and $\beta_X,\beta_y,\beta_{Xy},\beta_{yX}$, we can see that for all $\X,\tilde{\X}\in\Sn$ and all $\y,\tilde{\y}\in\mathcal{K}$
\begin{align*}
& \Vert(\nabla_{\X}f(\X,\y),-\nabla_{\y}f(\X,\y))-(\nabla_{\X}f(\tilde{\X},\tilde{\y}),-\nabla_{\y}f(\tilde{\X},\tilde{\y}))\Vert^2 
\\ & = \Vert\nabla_{\X}f(\X,\y)-\nabla_{\X}f(\tilde{\X},\v)\Vert_F^2 + \Vert\nabla_{\y}f(\X,\y)-\nabla_{\y}f(\tilde{\X},\tilde{\y})\Vert_2^2
\\ & \le 2\Vert\nabla_{\X}f(\X,\y)-\nabla_{\X}f(\tilde{\X},\y)\Vert_F^2 + 2\Vert\nabla_{\X}f(\tilde{\X},\y)-\nabla_{\X}f(\tilde{\X},\tilde{\y})\Vert_F^2 \\ &  \ \ \ + 2\Vert\nabla_{\y}f(\X,\y)-\nabla_{\y}f(\tilde{\X},\y)\Vert_2^2 + 2\Vert\nabla_{\y}f(\tilde{\X},\y)-\nabla_{\y}f(\tilde{\X},\tilde{\y})\Vert_2^2
\\ & \le 2(\beta_{X}^2+\beta_{yX}^2)\Vert\X-\tilde{\X}\Vert_F^2 + 2(\beta_{y}^2+\beta_{Xy}^2)\Vert\y-\tilde{\y}\Vert_2^2
\\ & \le 2\max\lbrace\beta_{X}^2+\beta_{yX}^2,\beta_{y}^2+\beta_{Xy}^2\rbrace\Vert(\X,\y)-(\tilde{\X},\tilde{\y})\Vert^2.
\end{align*}

Therefore, $\beta= \sqrt{2}\max\left\lbrace\sqrt{\beta_{X}^2+\beta_{yX}^2},\sqrt{\beta_{y}^2+\beta_{Xy}^2}\right\rbrace$.

The following lemma highlights a connection between the gradient of a saddle-point of \eqref{problem1} and subgradients of an optimal solution to \eqref{nonSmoothProblem} for which the first order optimality condition holds. One of the connections we will be interested in, is that GSC for Problem \eqref{nonSmoothProblem} implies GSC (with the same parameters) for Problem \eqref{problem1}. However, to prove this specific connection we require an additional structural assumption on the objective function $g(\cdot)$. We note that this assumption holds for all applications mentioned in this paper.
\begin{assumption}\label{ass:struct}
$g(\X)$ is of the form $g(\X) = h(\X) + \max_{\y\in\mK}\y^{\top}(\mA(\X)-\b)$, where $h(\cdot)$ is smooth and convex, and $\mA$ is a linear map.
\end{assumption}

\begin{lemma} \label{lemma:connectionSubgradientNonSmoothAndSaddlePoint}
If $(\X^*,\y^*)$ is a saddle-point of Problem \eqref{problem1} then $\X^*$ is an optimal solution to Problem \eqref{nonSmoothProblem},  $\nabla_{\X}f(\X^*,\y^*)\in\partial g(\X^*)$,  and for all $\X\in\Sn$ it holds that $\langle\X-\X^*,\nabla_{\X}f(\X^*,\y^*)\rangle\ge0$.
Conversely, under \cref{ass:struct}, if $\X^*$ is an optimal solution to Problem \eqref{nonSmoothProblem}, and $\G^*\in\partial{}g(\X^*)$ which satisfies $\langle\X-\X^*,\G^*\rangle\ge0$  for all $\X\in\Sn$, then there exists $\y^*\in\argmax_{\y\in\mathcal{K}}f(\X^*,\y)$ such that $(\X^*,\y^*)$ is a saddle-point of Problem \eqref{problem1}, and $\nabla_{\X}f(\X^*,\y^*)=\G^*$.
\end{lemma}

The proof is given in \cref{sec:proofSubgradEquiv}. 
The connection between the gradient of an optimal solution to the saddle-point problem and a subgradient of a corresponding optimal solution in the equivalent nonsmooth problem established in \cref{lemma:connectionSubgradientNonSmoothAndSaddlePoint}, naturally
leads to the formulation of the following generalized strict complementarity assumption for saddle-point problems.

\begin{assumption}[generalized strict complementarity for saddle-points]\label{ass:strictcompSaddlePoint}
We say a saddle-point $(\X^*,\y^*)\in\Sn\times\mathcal{K}$ of Problem \eqref{problem1} with $\rank(\X^*)=r^*$ satisfies the generalized strict complementarity assumption with parameters $r\geq r^*, \delta>0$, if $\lambda_{n-r}(\nabla_{\X}f(\X^*,\y^*))-\lambda_{n}(\nabla_{\X}f(\X^*,\y^*))\ge\delta$.
\end{assumption}

\begin{remark}
Note that under \cref{ass:struct}, due to Lemma \ref{lemma:connectionSubgradientNonSmoothAndSaddlePoint},  GSC with parameters $r,\delta$ for some optimal solution $\X^*$ to Problem \eqref{nonSmoothProblem} implies GSC with parameters $r,\delta$ to a corresponding saddle-point $(\X^*,\y^*)$ of Problem  \eqref{problem1}. Nevertheless, \cref{ass:struct} is not necessary for proving our convergence results for Problem \eqref{problem1}, which are directly stated in terms of Assumption \ref{ass:strictcompSaddlePoint}. %It is only necessary to translate the generalized strict complementarity assumption of Problem \eqref{nonSmoothProblem} (\cref{ass:strictcompNonSmooth}) to generalized strict complementarity of Problem \eqref{problem1} (\cref{ass:strictcompSaddlePoint}).
\end{remark}

\section{Projected Extragradient Method with Low-Rank Projections}

In this section we formally state and prove our main result: the projected extragradient method for the saddle-point Problem \eqref{problem1}, when initialized in the proximity of a saddle-point which satisfies GSC (\cref{ass:strictcompSaddlePoint}), converges with its original $O(1/t)$ rate while requiring only two low-rank SVD computations per iteration.

%The projected extragradient method is given in \cref{alg:EG}.
\begin{algorithm}[H]
	\caption{Projected extragradient method for saddle-point problems (see also \cite{extragradientK,Nemirovski})}\label{alg:EG}
	\begin{algorithmic}
		\STATE \textbf{Input:} sequence of step-sizes $\{\eta_t\}_{t\geq 1}$ 
		\STATE \textbf{Initialization:} $(\X_1,\y_1)\in\Sn\times\mathcal{K}$
		\FOR{$t = 1,2,...$} 
            \STATE $\Z_{t+1}=\Pi_{\Sn}[\X_t-\eta_t\nabla_{\X}f(\X_t,\y_t)]$   
			\STATE $\w_{t+1}=\Pi_{\mathcal{K}}[\y_t+\eta_t\nabla_{\y}f(\X_t,\y_t)]$
			\STATE $\X_{t+1}=\Pi_{\Sn}[\X_t-\eta_t\nabla_{\X}f(\Z_{t+1},\w_{t+1})]$   
			\STATE $\y_{t+1}=\Pi_{\mathcal{K}}[\y_t+\eta_t\nabla_{\y}f(\Z_{t+1},\w_{t+1})]$
        \ENDFOR
	\end{algorithmic}
\end{algorithm}

First, in the following lemma we state the standard convergence result of the projected extragradient method, which is a well known result\footnote{\cite{sebastienMD,Nemirovski} prove this result with respect to the ergodic series. A small adjustment of the proof proves the same with respect to the minimum and maximum iterates.}. For completeness we include the proof in \cref{appendixConvergenceProof}.

%The $O(1/t)$ convergence rate of the projected extragradient method is a well known result\footnote{\cite{Nemirovski} prove this result with respect to the ergodic series. A small change in the proof allows us to prove the same with respect to the minimum and maximum iterates.}. For completeness we include the proof in \cref{appendixConvergenceProof}.

\begin{lemma} \label{lemma:EGconvergenceRate}
Let $\lbrace(\X_t,\y_t)\rbrace_{t\ge1}$ and $\lbrace(\Z_{t},\w_t)\rbrace_{t\ge2}$ be the sequences generated by \cref{alg:EG} with a fixed step-size $\eta_t=\eta\le\min\left\lbrace\frac{1}{\beta_{X}+\beta_{Xy}},\frac{1}{\beta_{y}+\beta_{yX}},\frac{1}{\beta_{X}+\beta_{yX}},\frac{1}{\beta_{y}+\beta_{Xy}}\right\rbrace$ then
%\textcolor{red}{
\begin{align*}
\max_{\y\in\mathcal{K}} f\left(\frac{1}{T}\sum_{t=1}^T \Z_{t+1},\y\right) -  \min_{\X\in\Sn} f\left(\X,\frac{1}{T}\sum_{t=1}^T\w_{t+1}\right) & \le \frac{D^2}{2\eta T},
\end{align*}%}
where $D:=\sup_{(\X,\y),(\tilde{\X},\tilde{\y})\in{\Sn\times\mathcal{K}}}\Vert(\X,\y)-(\tilde{\X},\tilde{\y})\Vert$. 
%In particular,
%\begin{enumerate}
%\item $\min_{t\in[T]}\left(\max_{\y\in\mathcal{K}} f(\Z_{t+1},\y) - \min_{\X\in\Sn} f(\X,\w_{t+1})\right)
%\le \frac{D^2}{2\eta T}$.
%\item $\max_{\y\in\mathcal{K}} f\left(\frac{1}{T}\sum_{t=1}^T \Z_{t+1},\y\right) -  \min_{\X\in\Sn} f\left(\X,\frac{1}{T}\sum_{t=1}^T\w_{t+1}\right) \le \frac{D^2}{2\eta T}$.
%\end{enumerate}
\end{lemma}

%\subsection{Local convergence of the projected extragradient method using low-rank projections}

%We begin by stating our main theoretical result which proves that there exists an area around the optimal solution for which all projections onto the spectrahedron that are required in the projected extragradient method can be replaced with low-rank truncated projections, using only low-rank SVD computations.

We can now state our main theorem.

\begin{theorem}[main theorem] \label{theroem:allPutTogether}
Fix an optimal solution $(\X^*,\y^*)\in{\Sn\times\mathcal{K}}$ to Problem \eqref{problem1}. Let $\tilde{r}$ denote the multiplicity of $\lambda_{n}(\nabla_{\X}f(\X^*,\y^*))$ and for any $r\ge \tilde{r}$ define $\delta(r)=\lambda_{n-r}(\nabla_{\X}f(\X^*,\y^*)-\lambda_{n}(\nabla_{\X}f(\X^*,\y^*)$. Let $\lbrace(\X_t,\y_t)\rbrace_{t\ge1}$ and $\lbrace(\Z_{t},\w_t)\rbrace_{t\ge2}$ be the sequences of iterates generated by \cref{alg:EG} with a fixed step-size \[ \eta=\min\Bigg\lbrace\frac{1}{2\sqrt{\beta_X^2+\beta_{yX}^2}},\frac{1}{2\sqrt{\beta_y^2+\beta_{Xy}^2}},\frac{1}{\beta_{X}+\beta_{Xy}},\frac{1}{\beta_y+\beta_{yX}}\Bigg\rbrace. \] 
Assume the initialization $(\X_1,\y_1)$ satisfies $\Vert(\X_1,\y_1)-(\X^*,\y^*)\Vert\le R_0(r)$, where
\begin{align*}
R_0(r):= \frac{\eta}{(1+\sqrt{2})\left(1+(2+\sqrt{2})\eta\max\lbrace \beta_{X},\beta_{Xy}\rbrace\right)}\max\Bigg\{\frac{\sqrt{\tilde{r}}\delta(r-\tilde{r}+1)}{2},\frac{\delta(r)}{(1+1/\sqrt{\tilde{r}})}\Bigg\}.
\end{align*}
Then, for all $t\ge1$, the projections $\Pi_{\Sn}[\X_t-\eta\nabla_{\X}f(\X_t,\y_t)]$ and $\Pi_{\Sn}[\X_t-\eta\nabla_{\X}f(\Z_{t+1},\w_{t+1})]$ can be replaced with their rank-r truncated counterparts (see \eqref{truncatedProjection}) without changing the sequences $\lbrace(\X_t,\y_t)\rbrace_{t\ge1}$ and $\lbrace(\Z_{t},\w_t)\rbrace_{t\ge2}$, and for any $T\ge0$ it holds that
%\textcolor{red}{
\begin{align*}
& \max_{\y\in\mathcal{K}} f\left(\frac{1}{T}\sum_{t=1}^T \Z_{t+1},\y\right) -  \min_{\X\in\Sn} f\left(\X,\frac{1}{T}\sum_{t=1}^T\w_{t+1}\right)
\\ & \le \frac{D^2\max{\left\lbrace\sqrt{\beta_X^2+\beta_{yX}^2},\sqrt{\beta_y^2+\beta_{Xy}^2},\frac{1}{2}(\beta_{X}+\beta_{Xy}),\frac{1}{2}(\beta_y+\beta_{yX})\right\rbrace}}{T},
\end{align*} %}
where $D:=\sup_{(\X,\y),(\Z,\w)\in{\Sn\times\mathcal{K}}}\Vert(\X,\y)-(\Z,\w)\Vert$.
\end{theorem}

\begin{remark}
Note that \cref{theroem:allPutTogether} implies that if standard strict complementarity holds for Problem \eqref{problem1}, that is \cref{ass:strictcompSaddlePoint} holds with $r=r^*=\rank(\X^*)$ and some $\delta>0$, then only rank-$r^*$ SVDs are required  so that Algorithm \ref{alg:EG} converges with the guaranteed  convergence rate of $O(1/t)$, when initialized with a ``warm-start''. Furthermore, by using SVDs of rank $r>r^*$, with moderately higher values of $r$, we can increase the radius of the ball in which Algorithm \ref{alg:EG} needs to be initialized quite significantly.
\end{remark}

To prove \cref{theroem:allPutTogether} we first prove two technical lemmas. We begin by proving that the iterates of Algorithm \ref{alg:EG} always remain inside a ball of a certain radius around an optimal solution.

\begin{lemma} \label{lemma:convergenceOfXYZWsequences}
Let $\lbrace(\X_t,\y_t)\rbrace_{t\ge1}$ and $\lbrace(\Z_{t},\w_t)\rbrace_{t\ge2}$ be the sequences generated by \cref{alg:EG} with a step-size $\eta_t\le\frac{1}{\beta}$, and let $(\X^*,\y^*)$ be some optimal solution to Problem \eqref{problem1}. Then for all $t\ge1$ it holds that 
\begin{align*}
\Vert(\X_{t+1},\y_{t+1})-(\X^*,\y^*)\Vert & \le \Vert(\X_{t},\y_t)-(\X^*,\y^*)\Vert,
\\ \Vert(\Z_{t+1},\w_{t+1})-(\X^*,\y^*)\Vert & \le \Bigg(1+\frac{1}{\sqrt{1-\eta_{t}^2\beta^2}}\Bigg)\Vert(\X_t,\y_t)-(\X^*,\y^*)\Vert.
\end{align*}
\end{lemma}

\begin{proof}

A known inequality of the EG algorithm (see for example Lemma 12.1.10 in \cite{nonexpansiveProof}) is
\begin{align} \label{ineq:EGconvergenceOfSeries}
& \Vert(\X_{t+1},\y_{t+1})-(\X^*,\y^*)\Vert^2 \nonumber
\\& \le \Vert(\X_{t},\y_t)-(\X^*,\y^*)\Vert^2 - (1-\eta_t^2\beta^2)\Vert(\X_t,\y_t)-(\Z_{t+1},\w_{t+1})\Vert^2.
\end{align}

Since $\eta_t^2\beta^2\le1$ it follows that
\begin{align*}
\Vert(\X_{t+1},\y_{t+1})-(\X^*,\y^*)\Vert \le \Vert(\X_{t},\y_t)-(\X^*,\y^*)\Vert.
\end{align*}

In addition, using \eqref{ineq:EGconvergenceOfSeries}
\begin{align*}
\Vert(\X_t,\y_t)-(\Z_{t+1},\w_{t+1})\Vert & \le \sqrt{(1-\eta_t^2\beta^2)^{-1}}\Vert(\X_{t},\y_t)-(\X^*,\y^*)\Vert.
\end{align*}

Therefore,
\begin{align*}
\Vert(\Z_{t+1},\w_{t+1})-(\X^*,\y^*)\Vert & \le \Vert(\Z_{t+1},\w_{t+1})-(\X_{t},\y_{t})\Vert+\Vert(\X_{t},\y_{t})-(\X^*,\y^*)\Vert
\\ & \le \sqrt{(1-\eta_{t}^2\beta^2)^{-1}}\Vert(\X_{t},\y_{t})-(\X^*,\y^*)\Vert \\
& \ \ \ +\Vert(\X_{t},\y_{t})-(\X^*,\y^*)\Vert
\\ & = \Bigg(1+\frac{1}{\sqrt{1-\eta_{t}^2\beta^2}}\Bigg)\Vert(\X_{t},\y_{t})-(\X^*,\y^*)\Vert.
\end{align*}

\end{proof}

%Since all the iterates of \cref{alg:EG} are within some radius from the optimal solution, 
We now prove that when close enough to a low-rank saddle-point of Problem \eqref{problem1}, under an assumption of an eigen-gap in the gradient of the saddle-point, both projections onto the spectrahedron that are necessary in each iteration of \cref{alg:EG}, result in low-rank matrices. %This ensures that the projections throughout all the iterations will always return low-rank matrices.

\begin{lemma}
\label{lemma:radiusOfLowRankProjections}
Let $(\X^*,\y^*)$ be an optimal solution to Problem \eqref{problem1}. Let $\tilde{r}$ denote the multiplicity of $\lambda_{n}(\nabla_{\X}f(\X^*,\y^*))$ and for any $r\ge \tilde{r}$ denote $\delta(r):=\lambda_{n-r}(\nabla_{\X}f(\X^*,\y^*))-\lambda_{n}(\nabla_{\X}f(\X^*,\y^*))$.
Then, for any $\eta\ge0$ and $(\X,\y)\in\mathcal{S}_n\times\mathcal{K}$, if
\begin{align*}
& \Vert(\X,\y)-(\X^*,\y^*)\Vert
\\ & \le \frac{\eta}{1+\sqrt{2}\eta\max\lbrace \beta_{X},\beta_{Xy}\rbrace\left(1+\frac{1}{\sqrt{1-\eta^2\beta^2}}\right)}\max\left\lbrace\frac{\sqrt{\tilde{r}}\delta(r-\tilde{r}+1)}{2},\frac{\delta(r)}{(1+1/\sqrt{\tilde{r}})}\right\rbrace
\end{align*}
then $\rank\left( \Pi_{\mathcal{S}_n}[\X-\eta \nabla_{\X}f(\X,\y)]\right)\le r$ and $\rank\left( \Pi_{\mathcal{S}_n}[\X-\eta \nabla_{\X}f(\Z_+,\w_+)]\right)\le r$ where $\Z_+=\Pi_{\Sn}[\X-\eta\nabla_{\X}f(\X,\y)]$ and $\w_+=\Pi_{\mathcal{K}}[\y-\eta\nabla_{\y}f(\X,\y)]$. 
\end{lemma}

\begin{proof}
Denote $\P^*=\X^*-\eta\nabla_{\X}f(\X^*,\y^*)$. By \cref{lemma:connectionSubgradientNonSmoothAndSaddlePoint}, $\nabla_{\X}f(\X^*,\y^*)$ is a subgradient of the corresponding nonsmooth objective $g(\X)=\max_{\y\in\mathcal{K}}f(\X,\y)$ at the point $\X^*$. Moreover, this subgradient also satisfies the first-order optimality condition. Hence, invoking \cref{lemma:Xopt-subgradEigen} with this subgradient  we have that
\begin{align} \label{eq:eigsOfPstar}
 \forall i\le \rank(\X^*) & :\ \lambda_i(\P^*) = \lambda_i(\X^*)-\eta\lambda_n(\nabla_{\X}f(\X^*,\y^*)); \nonumber \\
\forall i>\rank(\X^*) & :\ \lambda_i(\P^*) = -\eta\lambda_{n-i+1}(\nabla_{\X}f(\X^*,\y^*)). 
\end{align}

Therefore, using \eqref{eq:eigsOfPstar} and the fact that $\lambda_{n-i+1}(\nabla{}f_{\X}(\X^*,\y^*)) = \lambda_{n}(\nabla{}f_{\X}(\X^*,\y^*))$ for all $i\leq \tilde{r}$ we have,
\begin{align} \label{ineq:sumOfXeigsInProof}
\sum_{i=1}^{\tilde{r}} {\lambda_i(\P^*)} & = \sum_{i=1}^{\tilde{r}} {\lambda_i(\X^*-\eta\nabla_{\X}f(\X^*,\y^*))}
= \sum_{i=1}^{\tilde{r}} \lambda_i(\X^*) -\eta\sum_{i=1}^{\tilde{r}} \lambda_{n-i+1}(\nabla_{\X}f(\X^*,\y^*)) \nonumber
\\ & = \sum_{i=1}^{\rank(\X^*)}\lambda_i(\X^*) -\eta \sum_{i=1}^{\tilde{r}}\lambda_n(\nabla_{\X}f(\X^*,\y^*)) = 1 -\eta \tilde{r}\lambda_n(\nabla_{\X}f(\X^*,\y^*)).
\end{align}

Let $\P\in\mathbb{S}^n$.  From the structure of the Euclidean projection onto the spectrahedron (see Eq. \eqref{eq:euclidProj}), it follows that
a sufficient condition so that $\rank\left( \Pi_{\mathcal{S}_n}[\P]\right)\le r$ is that $\sum_{i=1}^{r} {\lambda_i(\P)}-r\lambda_{r+1}(\P) \ge 1$. We will bound the LHS of this inequality.

First, it holds that
\begin{align} \label{ineq:sumTildeRbound}
\sum_{i=1}^{\tilde{r}} {\lambda_i(\P)} & \underset{(a)}{\ge} \sum_{i=1}^{\tilde{r}} {\lambda_i(\P^*)} - \sum_{i=1}^{\tilde{r}} {\lambda_i(\P^*-\P)} 
\ge \sum_{i=1}^{\tilde{r}} {\lambda_i(\P^*)} - \sqrt{\tilde{r}\sum_{i=1}^{\tilde{r}} {\lambda_i^2(\P-\P^*)}} \nonumber
\\ & \ge \sum_{i=1}^{\tilde{r}} {\lambda_i(\P^*)} - \sqrt{\tilde{r}\sum_{i=1}^{n} {\lambda_i^2(\P-\P^*)}} 
\ge \sum_{i=1}^{\tilde{r}} {\lambda_i(\P^*)} - \sqrt{\tilde{r}}\Vert\P-\P^*\Vert_F \nonumber
\\ & \underset{(b)}{\ge}  1 -\eta \tilde{r}\lambda_n(\nabla_{\X}f(\X^*,\y^*))- \sqrt{\tilde{r}}\Vert\P-\P^*\Vert_F,
\end{align}
where (a) holds from Ky Fan’s inequality for eigenvalues and (b) holds from \eqref{ineq:sumOfXeigsInProof}.

Now, for any $r\ge \tilde{r}$ using Weyl's inequality and \eqref{eq:eigsOfPstar}
\begin{align} \label{ineq:Rplus1bound1}
\lambda_{r+1}(\P) & \le \lambda_{r+1}(\P^*)+\lambda_{1}(\P-\P^*) \le \lambda_{r+1}(\P^*)+\Vert\P-\P^*\Vert_F \nonumber
\\ & = -\eta\lambda_{n-r}(\nabla_{\X}f(\X^*,\y^*))+\Vert\P-\P^*\Vert_F.
\end{align}

Thus, combining \eqref{ineq:sumTildeRbound} and \eqref{ineq:Rplus1bound1} we obtain
\begin{align} \label{ineq:opt1_Pbound}
\sum_{i=1}^{r} {\lambda_i(\P)}-r\lambda_{r+1}(\P)  \ge \sum_{i=1}^{\tilde{r}} {\lambda_i(\P)}-\tilde{r}\lambda_{r+1}(\P)
\ge 1 +\eta \tilde{r}\delta(r) - (\tilde{r}+\sqrt{\tilde{r}})\Vert\P-\P^*\Vert_F.
\end{align} 

Alternatively, if $r\ge2\tilde{r}-1$ then using the general Weyl inequality and \eqref{eq:eigsOfPstar} we obtain
\begin{align} \label{ineq:Rplus1bound2}
\lambda_{r+1}(\P) & \le \lambda_{r-\tilde{r}+2}(\P^*)+\lambda_{\tilde{r}}(\P-\P^*) = \lambda_{r-\tilde{r}+2}(\P^*)+\sqrt{\lambda^2_{\tilde{r}}(\P-\P^*)} \nonumber
\\ & \le \lambda_{r-\tilde{r}+2}(\P^*)+\frac{1}{\sqrt{\tilde{r}}}\Vert\P-\P^*\Vert_F
= -\eta\lambda_{n-r+\tilde{r}-1}(\nabla_{\X}f(\X^*,\y^*))+\frac{1}{\sqrt{\tilde{r}}}\Vert\P-\P^*\Vert_F.
\end{align} 

Thus, combining \eqref{ineq:sumTildeRbound} and \eqref{ineq:Rplus1bound2} we obtain
\begin{align} \label{ineq:opt2_Pbound}
& \sum_{i=1}^{r} {\lambda_i(\P)}-r\lambda_{r+1}(\P)  \ge \sum_{i=1}^{\tilde{r}} {\lambda_i(\P)}-\tilde{r}\lambda_{r+1}(\P)
\ge 1 +\eta \tilde{r}\delta(r-\tilde{r}+1) - 2\sqrt{\tilde{r}}\Vert\P-\P^*\Vert_F.
\end{align} 

Now we are left with bounding $\Vert\P-\P^*\Vert_F$. Note that by the smoothness of $f$, for any $(\X,\y)\in\Sn\times\mathcal{K}$ it holds that
\begin{align} \label{ineq:betaSmoothnessInProof}
& \Vert\nabla_{\X}f(\X,\y)-\nabla_{\X}f(\X^*,\y^*)\Vert_F  \nonumber
\\ & \le \Vert\nabla_{\X}f(\X,\y)-\nabla_{\X}f(\X^*,\y)\Vert_F+\Vert\nabla_{\X}f(\X^*,\y)-\nabla_{\X}f(\X^*,\y^*)\Vert_F \nonumber
\\ & \le \beta_{X}\Vert\X-\X^*\Vert_F+\beta_{Xy}\Vert\y-\y^*\Vert_2.
\end{align}

Taking $\P=\X-\eta\nabla_{\X}f(\X,\y)$ we get
\begin{align*}
\Vert\P-\P^*\Vert_F & = \Vert\X-\eta\nabla_{\X}f(\X,\y)-\X^*+\eta\nabla_{\X}f(\X^*,\y^*)\Vert_F
\\ & \le \Vert\X-\X^*\Vert_F+\eta\Vert\nabla_{\X}f(\X,\y)-\nabla_{\X}f(\X^*,\y^*)\Vert_F
\\ & \le \Vert(\X,\y)-(\X^*,\y^*)\Vert+\eta\Vert\nabla_{\X}f(\X,\y)-\nabla_{\X}f(\X^*,\y^*)\Vert_F
\\ & \le \Vert(\X,\y)-(\X^*,\y^*)\Vert+\eta\beta_{X}\Vert\X-\X^*\Vert_F+\eta\beta_{Xy}\Vert\y-\y^*\Vert_2,
\end{align*}
where the last inequality holds from \eqref{ineq:betaSmoothnessInProof}.

For any $a,b\ge0$ it holds that 
\begin{align*}
a\Vert\X-\X^*\Vert_F+b\Vert\y-\y^*\Vert_2 & \le \max\lbrace a,b\rbrace\left(\Vert\X-\X^*\Vert_F+\Vert\y-\y^*\Vert_2\right)
\\ & \le \sqrt{2}\max\lbrace a,b\rbrace \Vert(\X,\y)-(\X^*,\y^*)\Vert.
\end{align*}

Thus, by taking $a=\eta\beta_{X}$ and $b=\eta\beta_{Xy}$ we obtain
\begin{align} \label{ineq:p1NormBound}
\Vert\P-\P^*\Vert_F & \le \left(1+\sqrt{2}\eta\max\lbrace\beta_{X},\beta_{Xy}\rbrace\right)\Vert(\X,\y)-(\X^*,\y^*)\Vert.
\end{align}

Therefore, plugging \eqref{ineq:p1NormBound} into \eqref{ineq:opt1_Pbound} we obtain that the condition $\sum_{i=1}^{r} {\lambda_i(\P)}-r\lambda_{r+1}(\P)\ge1$ holds if
\begin{align*} 
\Vert(\X,\y)-(\X^*,\y^*)\Vert \le \frac{\eta\delta(r)}{(1+1/\sqrt{\tilde{r}})\left(1+\sqrt{2}\eta\max\lbrace\beta_{X},\beta_{Xy}\rbrace\right)}.
\end{align*}

Alternatively, plugging \eqref{ineq:p1NormBound} into \eqref{ineq:opt2_Pbound} we obtain that if $r\ge2\tilde{r}-1$ then the condition $\sum_{i=1}^{r} {\lambda_i(\P)}-r\lambda_{r+1}(\P)\ge1$ holds if
\begin{align*}
\Vert(\X,\y)-(\X^*,\y^*)\Vert \le \frac{\eta\sqrt{\tilde{r}}\delta(r-\tilde{r}+1)}{2\left(1+\sqrt{2}\eta\max\lbrace\beta_{X},\beta_{Xy}\rbrace\right)}.
\end{align*}

Note that $\delta(r-\tilde{r}+1)>0$ only if $r\ge2\tilde{r}-1$. Therefore, we can combine the last two inequalities to conclude that for any $r\ge\tilde{r}$ if
\begin{align} \label{ineq:XseriesRadius}
\Vert(\X,\y)-(\X^*,\y^*)\Vert \le \frac{\eta}{1+\sqrt{2}\eta\max\lbrace\beta_{X},\beta_{Xy}\rbrace}\max\Bigg\{\frac{\sqrt{\tilde{r}}\delta(r-\tilde{r}+1)}{2},\frac{\delta(r)}{(1+1/\sqrt{\tilde{r}})}\Bigg\}
\end{align}
then $\rank(\Pi_{\Sn}[\X-\nabla_{\X}f(\X,\y)])\le r$.

Similarly, taking $\P=\X-\eta\nabla_{\X}f(\Z_+,\w_+)$ we get
\begin{align*} 
\Vert\P-\P^*\Vert_F & = \Vert\X-\eta\nabla_{\X}f(\Z_+,\w_+)-\X^*+\eta\nabla_{\X}f(\X^*,\y^*)\Vert_F 
\\ & \le \Vert\X-\X^*\Vert_F+\eta\Vert\nabla_{\X}f(\Z_+,\w_+)-\nabla_{\X}f(\X^*,\y^*)\Vert_F
\\ & \le \Vert\X-\X^*\Vert_F+\eta\beta_{X}\Vert\Z_+-\X^*\Vert_F+\eta\beta_{Xy}\Vert\w_+-\y^*\Vert_2,
\end{align*}
where the last inequality holds from \eqref{ineq:betaSmoothnessInProof}.

For any $a,b,c\ge0$ it holds that 
\begin{align*}
& a\Vert\X-\X^*\Vert_F+b\Vert\Z_+-\X^*\Vert_F+c\Vert\w_+-\y^*\Vert_2 
\\ & \le a\Vert\X-\X^*\Vert_F+\max\lbrace b,c\rbrace\left(\Vert\Z_+-\X^*\Vert_F+\Vert\w_+-\y^*\Vert_2\right)
\\ & \le a\Vert\X-\X^*\Vert_F+\sqrt{2}\max\lbrace b,c\rbrace \Vert(\Z_+,\w_+)-(\X^*,\y^*)\Vert
\\ & \le a\Vert(\X,\y)-(\X^*,\y^*)\Vert+\sqrt{2}\max\lbrace b,c\rbrace \Vert(\Z_+,\w_+)-(\X^*,\y^*)\Vert
\\ & \le \left(a+\sqrt{2}\max\lbrace b,c\rbrace\left(1+\frac{1}{\sqrt{1-\eta^2\beta^2}}\right)\right)\Vert(\X,\y)-(\X^*,\y^*)\Vert,
\end{align*}
where the second to last inequality holds from \cref{lemma:convergenceOfXYZWsequences}.

Thus, by taking $a=1$, $b=\eta\beta_{X}$, and $c=\eta\beta_{Xy}$ we obtain
\begin{align} \label{ineq:p2NormBound}
\Vert\P-\P^*\Vert_F & \le \left(1+\sqrt{2}\eta\max\lbrace \beta_{X},\beta_{Xy}\rbrace\left(1+\frac{1}{\sqrt{1-\eta^2\beta^2}}\right)\right)\Vert(\X,\y)-(\X^*,\y^*)\Vert.
\end{align}

Therefore, plugging \eqref{ineq:p2NormBound} into \eqref{ineq:opt1_Pbound} we obtain that the condition $\sum_{i=1}^{r} {\lambda_i(\P)}-r\lambda_{r+1}(\P)\ge1$ holds if
\begin{align*}
\Vert(\X,\y)-(\X^*,\y^*)\Vert \le \frac{\eta\delta(r)}{(1+1/\sqrt{r^*})\left(1+\sqrt{2}\eta\max\lbrace \beta_{X},\beta_{Xy}\rbrace\left(1+\frac{1}{\sqrt{1-\eta^2\beta^2}}\right)\right)}.
\end{align*}

Alternatively, plugging \eqref{ineq:p2NormBound} into \eqref{ineq:opt2_Pbound} we obtain that if $r\ge2\tilde{r}-1$ then the condition $\sum_{i=1}^{r} {\lambda_i(\P)}-r\lambda_{r+1}(\P)\ge1$ holds if
\begin{align*}
\Vert(\X,\y)-(\X^*,\y^*)\Vert \le \frac{\eta\sqrt{\tilde{r}}\delta(r-\tilde{r}+1)}{1+\sqrt{2}\eta\max\lbrace \beta_{X},\beta_{Xy}\rbrace\left(1+\frac{1}{\sqrt{1-\eta^2\beta^2}}\right)}.
\end{align*}

Note that $\delta(r-\tilde{r}+1)>0$ only if $r\ge2\tilde{r}-1$. Therefore, we can combine the last two inequalities to conclude that for any $r\ge\tilde{r}$ if
\begin{align} \label{ineq:ZseriesRadius}
& \Vert(\X,\y)-(\X^*,\y^*)\Vert \nonumber
\\ & \le \frac{\eta}{1+\sqrt{2}\eta\max\lbrace \beta_{X},\beta_{Xy}\rbrace\left(1+\frac{1}{\sqrt{1-\eta^2\beta^2}}\right)}\max\left\lbrace\frac{\sqrt{\tilde{r}}\delta(r-\tilde{r}+1)}{2},\frac{\delta(r)}{(1+1/\sqrt{\tilde{r}})}\right\rbrace
\end{align}
then $\rank(\Pi_{\Sn}[\X-\nabla_{\X}f(\Z_+,\w_+)])\le r$.

Taking the minimum between \eqref{ineq:XseriesRadius} and \eqref{ineq:ZseriesRadius} gives us the bound on the radius in the lemma.

\end{proof}

Now we can prove \cref{theroem:allPutTogether}.

\begin{proof}[Proof of \cref{theroem:allPutTogether}]
We will prove by induction that for all $t\ge 1$ it holds that $\Vert(\X_t,\y_t)-(\X^*,\y^*)\Vert \le R_0(r)$ and $\Vert(\Z_{t},\w_{t})-(\X^*,\y^*)\Vert \le \left(1+\sqrt{2}\right)R_0(r)$, thus implying through  \cref{lemma:radiusOfLowRankProjections} that all projections $\Pi_{\Sn}[\X_t-\eta\nabla_{\X}f(\X_t,\y_t)]$ and $\Pi_{\Sn}[\X_t-\eta\nabla_{\X}f(\Z_{t+1},\w_{t+1})]$ can be replaced with their rank-r truncated counterparts given in \eqref{truncatedProjection}, without any change to the result. 

The initialization $\Vert(\X_1,\y_1)-(\X^*,\y^*)\Vert\le R_0(r)$ holds trivially.
Now, by \cref{lemma:convergenceOfXYZWsequences}, using recursion, we have that for all $t\ge1$,
\begin{align*}
\Vert(\X_{t+1},\y_{t+1})-(\X^*,\y^*)\Vert \le \Vert(\X_{t},\y_t)-(\X^*,\y^*)\Vert \le \cdots \le \Vert(\X_{1},\y_1)-(\X^*,\y^*)\Vert\le R_0(r),
\end{align*}
and for $\beta= \sqrt{2}\max\left\lbrace\sqrt{\beta_{X}^2+\beta_{yX}^2},\sqrt{\beta_{y}^2+\beta_{Xy}^2}\right\rbrace$ we have that,
\begin{align*}
\Vert(\Z_{t+1},\w_{t+1})-(\X^*,\y^*)\Vert & \le \Bigg(1+\frac{1}{\sqrt{1-\eta_{t}^2\beta^2}}\Bigg)\Vert(\X_t,\y_t)-(\X^*,\y^*)\Vert
\\ & \le \Bigg(1+\frac{1}{\sqrt{1-\eta_{t}^2\beta^2}}\Bigg)\Vert(\X_1,\y_1)-(\X^*,\y^*)\Vert
\\ & \le (1+\sqrt{2})\Vert(\X_1,\y_1)-(\X^*,\y^*)\Vert\le (1+\sqrt{2})R_0(r).
\end{align*}

Therefore, under the assumptions of the theorem, \cref{alg:EG} can be run using only rank-r truncated projections, while maintaining its original convergence rate stated in \cref{lemma:EGconvergenceRate}.
\end{proof}

%\begin{remark}
%Note that when applying Algorithm \ref{alg:EG} towards minimizing a nonsmooth objective of the form $g(\X) = \max_{\y\in\mK}f(\X,\y)$ (as discussed in Section \ref{sec:smooth2Saddle}), a guarantee of the form $\min_{t\in[T]} \max_{\y\in\mathcal{K}} f(\Z_{t+1},\y) - \max_{t\in[T]} \min_{\X\in\Sn}f(\X,\w_{t+1}) \leq \epsilon$, for some $\epsilon >0$, which is what we get from Theorem \ref{theroem:allPutTogether}, implies in particular that $\min_{t\in[T]}g(\X_t) - g(\X^*) \leq \epsilon$, where $\X^*$ is a minimizer of $g(\cdot)$ over $\mS_n$.
%\end{remark}

\begin{remark}
A downside of considering the saddle-point formulation \eqref{problem1} when attempting to solve Problem \eqref{nonSmoothProblem} that arises from Theorem \ref{theroem:allPutTogether}, is that not only do we need a ``warm-start'' initialization for the original primal matrix variable $\X$, in the saddle-point formulation we need a ``warm-start''  for the primal-dual pair $(\X,\y)$. Nevertheless, as we demonstrate extensively in Section \ref{sec:expr}, it seems that very simple initialization schemes work very well in practice.
\end{remark}

\subsection{Back to nonsmooth problems}

\begin{corollary} \label{cor:nonsmoothExtragradient}
Fix an optimal solution $\X^*\in\Sn$ to Problem \eqref{nonSmoothProblem} and assume \cref{ass:struct} holds. Let $\G^*\in\partial g(\X^*)$ which satisfies that $\langle\X-\X^*,\G^*\rangle\ge0$  for all $\X\in\Sn$. Let $\tilde{r}$ denote the multiplicity of $\lambda_{n}(\G^*)$ and for any $r\ge \tilde{r}$ define $\delta(r):=\lambda_{n-r}(\G^*)-\lambda_{n}(\G^*)$.
%Let $\y^*=\argmax_{\y\in\mathcal{K}}f(\X^*,\y)$ and 
Define $f$ as in Problem \eqref{problem1} and let $\lbrace(\X_t,\y_t)\rbrace_{t\ge1}$ and $\lbrace(\Z_{t},\w_t)\rbrace_{t\ge2}$ be the sequences of iterates generated by \cref{alg:EG} where $\eta$ and $R_0(r)$ are as defined in \cref{theroem:allPutTogether}.
Then, for all $t\ge1$ the projections $\Pi_{\Sn}[\X_t-\eta\nabla_{\X}f(\X_t,\y_t)]$ and $\Pi_{\Sn}[\X_t-\eta\nabla_{\X}f(\Z_{t+1},\w_{t+1})]$ can be replaced with rank-r truncated projections \eqref{truncatedProjection} without changing the sequences $\lbrace(\X_t,\y_t)\rbrace_{t\ge1}$ and $\lbrace(\Z_{t},\w_t)\rbrace_{t\ge2}$, and for any $T\ge0$ it holds that
%\textcolor{red}{
\begin{align*}
&g\left(\frac{1}{T}\sum_{t=1}^T \Z_{t+1}\right)-g(\X^*) 
\\ & \le \frac{D^2\max{\left\lbrace\sqrt{\beta_X^2+\beta_{yX}^2},\sqrt{\beta_y^2+\beta_{Xy}^2},\frac{1}{2}(\beta_{X}+\beta_{Xy}),\frac{1}{2}(\beta_y+\beta_{yX})\right\rbrace}}{T},
\end{align*}% }
where $D:=\sup_{(\X,\y),(\Z,\w)\in{\Sn\times\mathcal{K}}}\Vert(\X,\y)-(\Z,\w)\Vert$.
\end{corollary}

\begin{proof}
Since \cref{ass:struct} holds, invoking \cref{lemma:connectionSubgradientNonSmoothAndSaddlePoint} we obtain that there exists a point $\y^*\in\argmax_{\y\in\mathcal{K}}f(\X^*,\y)$ such that $(\X^*,\y^*)$ is a saddle-point of Problem \eqref{problem1}, and $\nabla_{\X}f(\X^*,\y^*)=\G^*$.
Therefore, the assumptions of \cref{theroem:allPutTogether} hold and so by \cref{theroem:allPutTogether} we get that
%\textcolor{red}{
\begin{align} \label{ineq:boundFromTheorem1}
& \max_{\y\in\mathcal{K}} f\left(\frac{1}{T}\sum_{t=1}^T \Z_{t+1},\y\right) -  \min_{\X\in\Sn} f\left(\X,\frac{1}{T}\sum_{t=1}^T\w_{t+1}\right) \nonumber
\\ & \le \frac{D^2\max{\left\lbrace\sqrt{\beta_X^2+\beta_{yX}^2},\sqrt{\beta_y^2+\beta_{Xy}^2},\frac{1}{2}(\beta_{X}+\beta_{Xy}),\frac{1}{2}(\beta_y+\beta_{yX})\right\rbrace}}{T}.
\end{align} 
From the definition of $g$ it holds that
\begin{align} \label{ineq:replaceMaxTerm}
g\left(\frac{1}{T}\sum_{t=1}^T \Z_{t+1}\right) =\max_{\y\in\mathcal{K}} f\left(\frac{1}{T}\sum_{t=1}^T \Z_{t+1},\y\right)
\end{align}
and
\begin{align} \label{ineq:replaceMinTerm}
\min_{\X\in\Sn} f\left(\X,\frac{1}{T}\sum_{t=1}^T\w_{t+1}\right) \le f\left(\X^*,\frac{1}{T}\sum_{t=1}^T\w_{t+1}\right) \le \max_{\y\in\mathcal{K}}f\left(\X^*,\y\right)=g(\X^*).
\end{align}
Plugging \eqref{ineq:replaceMaxTerm} and \eqref{ineq:replaceMinTerm} into the RHS of \eqref{ineq:boundFromTheorem1} we obtain the required result. 
%}

%\cref{lemma:connectionSubgradientNonSmoothAndSaddlePoint} offers the relationship between a subgradient of  Problem \eqref{nonSmoothProblem} and the gradient of a saddle-point in Problem \eqref{problem1}. The rest follows from \cref{theroem:allPutTogether}, where
%\begin{align*}
%\min_{t\in[T]}g(\Z_{t+1})=\min_{t\in[T]}\max_{\y\in\mathcal{K}}f(\Z_{t+1},\y) \le \frac{1}{T}\sum_{t=1}^T\max_{\y\in\mathcal{K}}f(\Z_{t+1},\y)
%\end{align*}
%is straightforward from the definition of $g$, and in addition, 
%\begin{align*}
%\frac{1}{T}\sum_{t=1}^T \min_{\X\in\Sn}f(\X,\w_{t+1}) \le \frac{1}{T}\sum_{t=1}^T f(\X^*,\w_{t+1}) \le \max_{\y\in\mathcal{K}}f\left(\X^*,\y\right)=g(\X^*).
%\end{align*}

%\begin{align*}
%\max_{t\in[T]} \min_{\X\in\Sn}f(\X,\w_{t+1}) \le \max_{t\in[T]} f(\X^*,\w_{t+1}) \le \max_{\y\in\mathcal{K}}f\left(\X^*,\y\right)=g(\X^*).
%\end{align*}

\end{proof}

\subsection{Efficiently-computable certificates for correctness of  low-rank projections}\label{sec:certificate}

Since Theorem \ref{theroem:allPutTogether} only applies in some neighborhood of an optimal solution, it is of interest to have a procedure for verifying if the rank-$r$ truncated projection of a given point indeed equals the exact Euclidean projection. In particular, from a practical point of view, it does not matter whether the conditions of Theorem \ref{theroem:allPutTogether} hold. In practice, as long as the truncated projection (see \eqref{truncatedProjection}) equals the exact projection (see \eqref{eq:euclidProj}), we are guaranteed that Algorithm \ref{alg:EG} converges correctly with rate $O(1/t)$, without needing to verify any other condition.
Luckily, the expression in \eqref{eq:euclidProj} which characterizes the structure of the Euclidean projection onto the spectrahedron, yields exactly such a verification procedure. As already noted in  \cite{garberLowRankSmooth}, for any $\X\in\mbS^n$, we have $\widehat{\Pi}_{\mathcal{S}_n}^r[\X]=\Pi_{\S_n}[\X]$ if an only if the condition
\begin{align*} %\label{cond:lowRankProjection}
\sum_{i=1}^r\lambda_i(\X)\ge 1+r\cdot\lambda_{r+1}(\X)
\end{align*}
holds. 
Note that verifying this condition  simply requires increasing the rank of the SVD computation by one, i.e., computing a rank-$(r+1)$ SVD of the matrix to project rather than a rank-$r$ SVD.

\section{Empirical Evidence}\label{sec:expr}

The goal of this section is to bring empirical evidence in support of our theoretical approach. We consider various tasks that take the form of minimizing a composite objective, i.e., the sum of a smooth convex function and a nonsmooth convex function, where the nonsmoothness comes from either an $\ell_1$-norm or $\ell_2$-norm regularizer / penalty term, over a $\tau$-scaled spectrahedron. In all cases the nonsmooth objective can be written as a saddle-point with function $f(\X,\y)$ which is linear in $\y$ and in particular satisfies Assumption \ref{ass:struct}. 

The tasks considered include 1. sparse PCA, 2. robust PCA, 3. low-rank and sparse recovery, 4. phase synchronization, and 5. linearly-constrained low-rank estimation, under variety of parameters.

For all tasks considered we generate random instances, and examine the sequences of iterates generated by Algorithm \ref{alg:EG} $\lbrace(\X_t,\y_t)\rbrace_{t\ge1}$, $\lbrace(\Z_{t},\w_t)\rbrace_{t\ge2}$, when initialized with simple initialization procedures. Out of both sequences generated, we choose our candidate for the optimal solution to be the iterate for which the dual-gap, which is a certificate for optimality, is smallest. See \cref{appendixDualGapCalculation}.

%%%%%%%%
\begin{comment}
For a given candidate  $(\widehat{\Z},\widehat{\w})$ for an approximated saddle-point, the dual-gap is the RHS of the following Eq. \eqref{ineq:dualGapBound}:
\begin{align} \label{ineq:dualGapBound}
& g(\widehat{\Z})-g^* \leq \max_{\y\in\mathcal{K}}f(\widehat{\Z},\y)-\min_{\substack{\trace(\X)=\tau,\\ \X\succeq 0}}f(\X,\widehat{\w}) \nonumber
\\ & \le \max_{\substack{\trace(\X)=\tau,\\ \X\succeq 0}}\langle\widehat{\Z}-\X,\nabla_{\X}f(\widehat{\Z},\widehat{\w})\rangle - \min_{\y\in\mathcal{K}} \langle\widehat{\w}-\y,\nabla_{\y}f(\widehat{\Z},\widehat{\w})\rangle.
\end{align} 
We prove  the correctness of the last inequality in \cref{appendixDualGapCalculation}.% that this is indeed a dual-gap for saddle-point problems.
\end{comment}
%%%%%%%%%%%%%%%%%

%It is easy to see that the maximizer of the first term in the RHS of \eqref{ineq:dualGapBound} is $\tau\v_n\v_n^{\top}$ where $\v_n$ is the smallest eigenvector of $\nabla_{\X}f(\widehat{\Z},\widehat{\W})$, and the minimizer of the second term is $\Y_{i,j}=\sign(\nabla_{\Y}f(\widehat{\Z},\widehat{\W})_{i,j})$ for $\mathcal{K}=\lbrace\Y\in\reals^{n\times n}\ \vert\ \Vert\Y\Vert_{\infty}\le1\rbrace$ and $\nabla_{\y}f(\widehat{\Z},\widehat{\w})/\Vert\nabla_{\y}f(\widehat{\Z},\widehat{\w})\Vert_2$ for $\mathcal{K}=\lbrace\y\in\reals^{n}\ \vert\ \Vert\y\Vert_{2}\le1\rbrace$. 

In all tasks considered the goal is to recover a ground-truth low-rank matrix $\M_0\in\mathbb{S}^n$ from some noisy observation of it $\M=\M_0+\N$, where $\N\in\mathbb{S}^{n}$ is a noise matrix. We measure the signal-to-noise ratio (SNR) as $\Vert\M_0\Vert_F^2\big/\Vert\N\Vert_F^2$. 
In all experiments we measure the relative initialization error by $\left\Vert\frac{\trace(\M_0)}{\tau}\X_1-\M_0\right\Vert_F^2\Big{/}\left\Vert\M_0\right\Vert_F^2$, and similarly we measure the relative recovery error by $\left\Vert\frac{\trace(\M_0)}{\tau}\X^*-\M_0\right\Vert_F^2\Big{/}\left\Vert\M_0\right\Vert_F^2$. Note that in some of the experiments we  take $\tau<\trace(\M_0)$ to prevent the method from overfitting the noise. In addition, we measure the (standard) strict complementarity parameter which corresponds to the eigen-gap $\textrm{gap}(\nabla_{\X}f(\X^*,\y^*)):=\lambda_{n-r}(\nabla_{\X}f(\X^*,\y^*))-\lambda_{n}(\nabla_{\X}f(\X^*,\y^*))$, $r = \rank(\M_0)$.

In all experiments we use SVDs of rank $r=\rank(\M_0)$ to compute the projections in Algorithm \ref{alg:EG} according to the truncated projection given in \eqref{truncatedProjection}. To certify the correctness of these low-rank projections (that is, that they equal the exact Euclidean projection) we confirm that the inequality
\begin{align*} 
\sum_{i=1}^r\lambda_i(\P_j)\ge \tau+r\cdot\lambda_{r+1}(\P_j)
\end{align*}
always holds for $\P_1=\X_t-\eta\nabla_{\X}f(\X_t,\Y_t)$ and $\P_2=\X_t-\eta\nabla_{\X}f(\Z_{t+1},\W_{t+1})$ (see also Section \ref{sec:certificate}). Indeed, we can now already state our main observation from the experiments:

\framebox{\parbox{\dimexpr\linewidth-2\fboxsep-2\fboxrule}{\itshape%
In all tasks considered and for all random instances generated, throughout all iterations of Algorithm \ref{alg:EG}, when initialized with a simple ``warm-start'' strategy and when computing only rank-$r$ truncated projections, $r=\rank(\M_0)$, the truncated projections of $\P_1=\X_t-\eta\nabla_{\X}f(\X_t,\Y_t)$ and $\P_2=\X_t-\eta\nabla_{\X}f(\Z_{t+1},\W_{t+1})$ equal their exact full-rank counterparts. That is, Algorithm \ref{alg:EG}, using only rank-$r$ SVDs, computed exactly the same sequences of iterates it would have computed if using full-rank SVDs.
}}

Aside from the above observation, in the sequel we demonstrate that all models considered indeed satisfy that: 1. the returned solution, denoted $(\X^*,\y^*)$, is of the same rank as the ground-truth matrix and satisfies the strict complementarity condition with non-negligible parameter  (measured by the eigengap $\lambda_{n-r}(\nabla_{\X}f(\X^*,\y^*))-\lambda_{n}(\nabla_{\X}f(\X^*,\y^*))$), 2. the recovery error of the returned solution indeed improves significantly over the error of the initialization point.

\subsection{Sparse PCA}

We consider the sparse PCA problem in a well known convex formulation taken from \cite{d2007direct} and its equivalent saddle-point formulation:
\begin{align*}
\min_{\substack{\trace(\X) =1,\\ \X\succeq 0}} \langle{\X,-\M}\rangle + \lambda\Vert{\X}\Vert_1 =
\min_{\substack{\trace(\X) =1,\\ \X\succeq 0}} \max_{\Vert\Y\Vert_{\infty}\le1}\lbrace\langle\X,-\M\rangle+\lambda\langle\X,\Y\rangle\rbrace,
\end{align*}
where $\M = \z\z^{\top} + \frac{c}{2}(\N+\N^{\top})$ is a noisy observation of a rank-one matrix $\z\z^{\top}$, with $\z$ being a sparse unit vector. Each entry $\z_i$ is chosen to be $0$ with probability $0.9$ and $U\lbrace1,\ldots,10\rbrace$ with probability $0.1$, and then we normalize $\z$ to be of unit norm.

We test the results obtained when adding different magnitudes of Gaussian or uniform noise. %For the sake of comparison between the different dimensions, 
We set the signal-to-noise ratio (SNR) to be a constant. Thus, we set the noise level to $c=\frac{2}{\textrm{SNR}\cdot\Vert\N+\N^{\top}\Vert_F}$ for our choice of $\textrm{SNR}$.

We initialize the $\X$ variable with the rank-one approximation of $\M$. That is, we take $\X_1=\u_1\u_1^{\top}$, where $\u_1$ is the top eigenvector of $\M$. For the $\Y$ variable we initialize it with $\Y_1=\sign(\X_1)$ which is a subgradient of $\Vert\X_1\Vert_1$.

We set the step-size to $\eta=1/(2\lambda)$ and we set the number of iterations to $T=1000$ and for any set of parameters we average the measurements over $10$ i.i.d. runs.

\begin{table*}[!htb]\renewcommand{\arraystretch}{1.3}
{\footnotesize
\begin{center}
  \begin{tabular}{| l | c | c | c | c |} \hline 
  dimension (n) & $100$ & $200$ & $400$ & $600$ \\ \hline
  \multicolumn{5}{|c|}{$\downarrow$ $\N\sim U[0,1]$, $\textrm{SNR}=1$ $\downarrow$}\\ \hline
  %noise level (c) & $0.0185$ & $0.0093$ & $0.0046$ & $0.0031$ \\ \hline
  $\lambda$ & $0.008$ & $0.004$ & $0.002$ & $0.0013$ \\ \hline
  initialization error & $0.5997$ & $0.6009$ & $0.5990$ & $0.6002$ \\ \hline
    recovery error & $0.0054$ & $0.0040$ & $0.0035$ & $0.0043$ \\ \hline
    dual gap & $4.1\times{10}^{-5}$ & $7.9\times{10}^{-5}$ & $4.9\times{10}^{-5}$ & $3.4\times{10}^{-6}$ \\ \hline
  $\lambda_{n-1}(\nabla_{\X}f(\X^*,\y^*))-\lambda_{n}(\nabla_{\X}f(\X^*,\y^*))$ & $0.8840$ & $0.8898$ & $0.8938$ & $0.8777$ \\ \hline 
  
  \multicolumn{5}{|c|}{$\downarrow$ $\N\sim U[0,1]$, $\textrm{SNR}=0.05$ $\downarrow$}\\ \hline
  %noise level (c) & $0.0827$ & $0.0414$ & $0.0207$ & $0.0138$ \\ \hline
  $\lambda$ & $0.04$ & $0.02$ & $0.01$ & $0.0067$ \\ \hline
  initialization error & $1.7456$ & $1.7494$ & $1.7566$ & $1.7625$ \\ \hline
  recovery error& $0.0425$ & $0.0244$ & $0.0149$ & $0.0100$ \\ \hline
    dual gap & $2.0\times{10}^{-9}$ & $5.8\times{10}^{-6}$ & $4.5\times{10}^{-4}$ & $0.0018$ \\ \hline
  $\lambda_{n-1}(\nabla_{\X}f(\X^*,\y^*))-\lambda_{n}(\nabla_{\X}f(\X^*,\y^*))$ & $0.7092$ & $0.7854$ & $0.8340$ & $0.8622$ \\ \hline 
  
  \multicolumn{5}{|c|}{$\downarrow$ $\N\sim\mathcal{N}(0.5,\I_n)$, $\textrm{SNR}=1$ $\downarrow$}\\ \hline
  %noise level (c) & $0.0115$ & $0.0058$ & $0.0029$ & $0.0019$ \\ \hline
  $\lambda$ & $0.006$ & $0.003$ & $0.0015$ & $0.001$ \\ \hline
  initialization error & $0.1584$ & $0.1464$ & $0.1443$ & $0.1411$ \\ \hline
  recovery error & $0.0059$ & $0.0033$ & $0.0019$ & $0.0015$ \\ \hline
    dual gap & $8.6\times{10}^{-4}$ & $0.0031$ & $0.0053$ & $0.0060$ \\ \hline
  % $\left\Vert\X^*\right\Vert_1$ & $9.6755$ & $17.6205$ & $34.9582$ & $50.4909$ \\ \hline
  $\lambda_{n-1}(\nabla_{\X}f(\X^*,\y^*))-\lambda_{n}(\nabla_{\X}f(\X^*,\y^*))$ & $0.8406$ & $0.8869$ & $0.9178$ & $0.9331$ \\ \hline 
  
  \multicolumn{5}{|c|}{$\downarrow$ $\N\sim\mathcal{N}(0.5,\I_n)$, $\textrm{SNR}=0.05$ $\downarrow$}\\ \hline
    %noise level (c) & $0.0513$ & $0.0258$ & $0.0129$ & $0.0086$ \\ \hline
  $\lambda$ & $0.04$ & $0.02$ & $0.01$ & $0.005$ \\ \hline
  initialization error & $1.6701$ & $1.6620$ & $1.6542$ & $1.6610$ \\ \hline
  recovery error & $0.0502$ & $0.0234$ & $0.0137$ & $0.0109$ \\ \hline
    dual gap & $1.9\times{10}^{-5}$ & $0.0041$ & $0.0534$ & $0.0409$ \\ \hline
  $\lambda_{n-1}(\nabla_{\X}f(\X^*,\y^*))-\lambda_{n}(\nabla_{\X}f(\X^*,\y^*))$ & $0.2200$ & $0.4076$ & $0.5460$ & $0.6788$ \\ \hline
   \end{tabular}
  \caption{Numerical results for the sparse PCA problem.
  }\label{table:sparsePCAuniformBigNoise}
\end{center}
}
\vskip -0.2in
\end{table*}\renewcommand{\arraystretch}{1}

\subsection{Low-rank and sparse matrix recovery}

We consider the problem of recovering a simultaneously low-rank and sparse covariance matrix \cite{lowRankAndSparse}, which can be written as the following saddle-point optimization problem:
\begin{align*}
\min_{\substack{\trace(\X) =1,\\ \X\succeq 0}}\frac{1}{2}\Vert{\X-\M}\Vert_F^2 + \lambda\Vert{\X}\Vert_1 =
\min_{\substack{\trace(\X) =\tau,\\ \X\succeq 0}} \max_{\Vert\Y\Vert_{\infty}\le1}\frac{1}{2}\Vert\X-\M\Vert_F^2+\lambda\langle\X,\Y\rangle,
\end{align*} 
where $\M = {\Z_0}{\Z_0}^{\top} + \frac{c}{2}(\N+\N^{\top})$ is a noisy observation of some low-rank and sparse covariance matrix ${\Z_0}{\Z_0}^{\top}$. We choose $\Z_0\in\mathbb{R}^{n\times r}$ to be a sparse matrix where each entry ${\Z_0}_{i,j}$ is chosen to be $0$ with probability $0.9$ and $U\lbrace1,\ldots,10\rbrace$ with probability $0.1$, and then we normalize $\Z_0$ to be of unit Frobenius norm. We choose $\N\sim\mathcal{N}(0.5,\I_n)$.

We test the model with $\rank(\Z_0\Z_0^{\top})=1,5,10$. %For the sake of comparison between the different dimensions, 
We set the signal-to-noise ratio (SNR) to be a constant and set the noise level to $c=\frac{2\Vert\Z_0\Z_0^{\top}\Vert_F}{\textrm{SNR}\cdot\Vert\N+\N^{\top}\Vert_F}$ for our choice of $\textrm{SNR}$. 

We initialize the $\X$ variable with the rank-r approximation of $\M$. That is, we take $\X_1=\U_r\diag\left(\Pi_{\Delta_{\tau,r}}[\diag(-\Lambda_r)]\right)\U_r^{\top}$, where $\U_r\Lambda_r\U_r^{\top}$ is the rank-r eigen-decomposition of $\M$ and $\Delta_{\tau,r}=\lbrace\z\in\reals^r~|~\z\ge0,\ \sum_{i=1}^r \z_i=\tau\rbrace$ is the simplex of radius $\tau$ in $\reals^r$. For the $\Y$ variable we initialize it with $\Y_1=\sign(\X_1)$ which is a subgradient of $\Vert\X_1\Vert_1$.  

We set the step-size to $\eta=1$, $\tau=0.7\cdot\trace(\Z_0\Z_0^{\top})$, and the number of iterations in each experiment to $T=2000$. For each value of $r$ and $n$ we average the measurements over over $10$ i.i.d. runs.

\begin{table*}[!htb]\renewcommand{\arraystretch}{1.3}
{\footnotesize
\begin{center}
  \begin{tabular}{| l | c | c | c | c | c | c | c | c | c |} \hline 
  dimension (n) & $100$ & $200$ & $400$ & $600$ \\ \hline
  
  \multicolumn{5}{|c|}{$\downarrow$ $r=\rank(\Z_0\Z_0^{\top})=1$, $\textrm{SNR}=0.48$ $\downarrow$}\\ \hline
  $\lambda$ & $0.0012$ & $0.0035$ & $0.0016$ & $0.001$ \\ \hline
%noise level (c) & $0.0166$ & $0.0083$ & $0.0042$ & $0.0028$ \\ \hline
  initialization error & $0.4562$ & $0.4471$ & $0.4507$ & $0.4450$ \\ \hline
    recovery error & $0.0364$ & $0.0193$ & $0.0160$ & $0.0168$ \\ \hline
    dual gap & $0.0083$ & $0.0086$ & $0.0020$ & $4.2\times{10}^{-4}$ \\ \hline
  $\lambda_{n-r}(\nabla_{\X}f(\X^*,\y^*))-\lambda_{n}(\nabla_{\X}f(\X^*,\y^*))$ & $0.0628$ & $0.1439$ & $0.1258$ & $0.1069$ \\ \hline

\multicolumn{5}{|c|}{$\downarrow$ $r=\rank(\Z_0\Z_0^{\top})=5$, $\textrm{SNR}=2.4$ $\downarrow$}\\ \hline
  $\lambda$ & $0.0012$ & $0.0006$ & $0.0003$ & $0.0002$ \\ \hline
%noise level (c) & $0.0036$ & $0.0018$ & $8.6\times{10}^{-4}$ & $5.6\times{10}^{-5}$ \\ \hline
  initialization error & $0.2132$ & $0.2103$ & $0.1983$ & $0.1907$ \\ \hline
    recovery error & $0.0641$ & $0.0478$ & $0.0349$ & $0.0274$ \\ \hline
    dual gap & $9.0\times{10}^{-4}$ & $4.3\times{10}^{-4}$ & $1.4\times{10}^{-4}$ & $7.3\times{10}^{-5}$ \\ \hline
  $\lambda_{n-r}(\nabla_{\X}f(\X^*,\y^*))-\lambda_{n}(\nabla_{\X}f(\X^*,\y^*))$ & $0.0148$ & $0.0200$ & $0.0257$ & $0.0277$ \\ \hline 

\multicolumn{5}{|c|}{$\downarrow$ $r=\rank(\Z_0\Z_0^{\top})=10$, $\textrm{SNR}=4.8$ $\downarrow$}\\ \hline
  $\lambda$ & $0.0007$ & $0.0004$ & $0.0002$ & $0.0001$ \\ \hline
%noise level (c) & $0.0019$ & $8.9\times{10}^{-4}$ & $4.3\times{10}^{-4}$ & $2.8\times{10}^{-4}$ \\ \hline
  initialization error & $0.1855$ & $0.1661$ & $0.1527$ & $0.1473$ \\ \hline
    recovery error & $0.0702$ & $0.0403$ & $0.0268$ & $0.0356$ \\ \hline
    dual gap & $4.9\times{10}^{-4}$ & $6.6\times{10}^{-4}$ & $4.2\times{10}^{-4}$ & $3.4\times{10}^{-5}$ \\ \hline
  $\lambda_{n-r}(\nabla_{\X}f(\X^*,\y^*))-\lambda_{n}(\nabla_{\X}f(\X^*,\y^*))$ & $0.0072$ & $0.0142$ & $0.0187$ & $0.0160$ \\ \hline
   \end{tabular}
  \caption{Numerical results for the low-rank and sparse matrix recovery problem.
  }\label{table:lowRank&SparseRank10}
\end{center}
}
\vskip -0.2in
\end{table*}\renewcommand{\arraystretch}{1}

\subsection{Robust PCA}

We consider the robust PCA problem \cite{robustPCA1} in the following formulation:
\begin{align*}
\min_{\substack{\trace(\X) =\tau,\\ \X\succeq 0}}\Vert{\X-\M}\Vert_1 =
\min_{\substack{\trace(\X) =\tau,\\ \X\succeq 0}} \max_{\Vert\Y\Vert_{\infty}\le1}\langle\X-\M,\Y\rangle,
\end{align*}
where $\M = r\Z_0\Z_0^{\top} + \frac{1}{2}(\N+\N^{\top})$ is a sparsely-corrupted observation of some rank-r  matrix $\Z_0\Z_0^{\top}$. We choose $\Z_0\in\mathbb{R}^{n\times r}$ to be a random unit Frobenius norm matrix. For $\N\in\reals^{n\times n}$, we choose each entry to be $0$ with probability $1-1/\sqrt{n}$ and otherwise $1$ or $-1$ with equal probability.

We initialize the $\X$ variable with the projection $\X_1=\Pi_{\lbrace\trace(\X)=\tau,\ \X\succeq 0\rbrace}[\M]$, and the $\Y$ variable with $\Y_1=\sign(\X_1-\M)$. 

%It is easy to see that for this problem $\beta_{X}=\beta_{y}=0$ and $\beta_{Xy}=\beta_{yX}=1$. Thus, $\beta= \sqrt{2}$. 

We test the model with $\rank(\Z_0\Z_0^{\top})=1,5,10$. For $\rank(\Z_0\Z_0^{\top})=1$ we set the step-size to $\eta=n/10$ and for $\rank(\Z_0\Z_0^{\top})=5,10$ we set it to $\eta=1$. We set the trace bound to $\tau=0.95\cdot\trace(r\Z_0\Z_0^{\top})$. For every set of parameters we average the measurements over $10$ i.i.d. runs. 

%\textcolor{red}{We vary the rank of $\Z_0$, which in turn determines the step-size $\eta$.  We set the trace bound to $\tau=0.95\cdot\trace(r\Z_0\Z_0^{\top})$, and the number of iterations to $T=30,000$. For every set of parameters we average the measurements over $10$ i.i.d. runs.}

\begin{table*}[!htb]\renewcommand{\arraystretch}{1.3}
{\footnotesize
\begin{center}
  \begin{tabular}{| l | c | c | c | c | c | c | c | c | c |} \hline 
  dimension (n) & $100$ & $200$ & $400$ & $600$ \\ \hline
  
 \multicolumn{5}{|c|}{$\downarrow$ $r=\rank(\Z_0\Z_0^{\top})=1$, $T=3000$ $\downarrow$}\\ \hline 
SNR & $0.0021$ & $7.2\times{10}^{-4}$ & $2.5\times{10}^{-4}$ & $1.3\times{10}^{-4}$ \\ \hline
  initialization error & $1.3511$ & $1.3430$ & $1.2889$ & $1.2606$ \\ \hline
    recovery error & $0.0084$ & $0.0107$ & $0.0109$ & $0.0107$ \\ \hline
    dual gap & $0.0016$ & $0.0029$ & $0.0044$ & $0.0069$ \\ \hline
  $\lambda_{n-r}(\nabla_{\X}f(\X^*,\y^*))-\lambda_{n}(\nabla_{\X}f(\X^*,\y^*))$ & $15.5944$ & $41.2139$ & $85.8117$ & $140.5349$ \\ \hline 

\multicolumn{5}{|c|}{$\downarrow$ $r=\rank(\Z_0\Z_0^{\top})=5$, $T=20,000$ $\downarrow$}\\ \hline
SNR & $0.0110$ & $0.0038$ & $0.0013$ & $6.9\times{10}^{-4}$ \\ \hline
  initialization error & $1.5501$ & $1.5527$ & $1.5221$ & $1.4833$ \\ \hline
    recovery error & $0.0092$ & $0.0092$ & $0.0087$ & $0.0075$  \\ \hline
    dual gap & $0.0084$ & $0.0390$ & $0.1866$ & $0.4721$ \\ \hline
  $\lambda_{n-r}(\nabla_{\X}f(\X^*,\y^*))-\lambda_{n}(\nabla_{\X}f(\X^*,\y^*))$ & $7.6734$ & $26.2132$ & $66.1113$ & $108.7215$  \\ \hline 

\multicolumn{5}{|c|}{$\downarrow$ $r=\rank(\Z_0\Z_0^{\top})=10$, $T=30,000$ $\downarrow$}\\ \hline
  SNR & $0.0229$ & $0.0077$ & $0.0026$ & $0.0014$ \\ \hline
  initialization error & $1.5729$ & $1.6485$ & $1.6317$ & $1.5949$ \\ \hline
  recovery error & $0.0079$ & $0.0081$ & $0.0073$ & $0.0065$  \\ \hline
  dual gap & $0.0139$ & $0.0338$ & $0.1533$ & $0.3561$ \\ \hline
  $\lambda_{n-r}(\nabla_{\X}f(\X^*,\y^*))-\lambda_{n}(\nabla_{\X}f(\X^*,\y^*))$ & $1.7945$ & $16.9890$ & $48.9799$ & $82.2727$  \\ \hline 
   \end{tabular}
  \caption{Numerical results for the robust PCA problem.
  }\label{table:robustPCArank10}
\end{center}
}
\vskip -0.2in
\end{table*}\renewcommand{\arraystretch}{1}

\subsection{Phase synchronization}

We consider the phase synchronization problem (see for instance \cite{phaseSyncronization1}) which can be written as:
\begin{align} \label{phaseSync}
\max_{\substack{\z\in\mathbb{C}^n ,\\ \vert z_j\vert=1\ \forall j\in[n]}}\z^*\M\z,
\end{align}
where $\M = \z_0\z_0^{*} + c\N$ is a noisy observation of some rank-one matrix $\z_0\z_0^{*}$ such that $\z_0\in\mathbb{C}^n$ and ${\z_0}_j=e^{i\theta_j}$ where $\theta_j\in[0,2\pi]$. We follow the statistical model in \cite{phaseSyncronization1} where the noise matrix $\N\in\mathbb{C}^{n\times{}n}$ is chosen such that every entry is
\begin{align*}
\N_{jk}=\Bigg\lbrace\begin{array}{lc}\mathcal{N}(0,1)+i\mathcal{N}(0,1) & j<k
\\ \overline{\N}_{kj} & j>k 
\\ 0 & j=k\end{array}.
\end{align*}
It is known that for a large $n$ and $c=\mathcal{O}\left(\sqrt{\frac{n}{\log{}n}}\right)$, with high probability the SDP relaxation of \eqref{phaseSync} is able to recover the original signal (see \cite{phaseSyncronization1}).

We solve a penalized version of the SDP relaxation of \eqref{phaseSync} which can be written as the following saddle-point optimization problem:
\begin{align*}
\min_{\substack{\trace(\X) =n,\\ \X\succeq 0}}\langle{\X,-\M}\rangle + \lambda\Vert{\textrm{diag}(\X)-\overrightarrow{\mathbf{1}}}\Vert_2 =
\min_{\substack{\trace(\X)=n,\\ \X\succeq 0}} \max_{\Vert\y\Vert_2\le1}\langle\X,-\M\rangle+\lambda\langle\diag(\X)-\overrightarrow{\mathbf{1}},\y\rangle,
\end{align*}
where $\overrightarrow{\mathbf{1}}$ is the all-ones vector.

While the phase synchronization problem is formulated over the complex numbers, extending our results to handle this model is straightforward.

We initialize the $\X$ variable with the rank-one approximation of $\M$. That is, we take $\X_1=n\u_1\u_1^{*}$, where $\u_1$ is the top eigenvector of $\M$. For the $\y$ variable we initialize it with $\y_1=(\diag(\X_1)-\overrightarrow{\mathbf{1}})/\Vert\diag(\X_1)-\overrightarrow{\mathbf{1}}\Vert_2$.

We set the noise level to $c=0.18\sqrt{n}$. We set the number of iterations in each experiment to $T=10,000$ and for each choice of $n$ we average the measurements over $10$ i.i.d. runs.

\begin{table*}[!htb]\renewcommand{\arraystretch}{1.3}
{\footnotesize
\begin{center}
  \begin{tabular}{| l | c | c | c | c | c | c | c | c | c |} \hline 
  dimension (n) & $100$ & $200$ & $400$ & $600$ \\ \hline
  SNR & $0.1553$ & $0.0775$ & $0.0387$ & $0.0258$ \\ \hline
$\lambda$ & $200$ & $600$ & $1600$ & $2800$ \\ \hline
$\eta$ & $1/400$ & $1/800$ & $1/1800$ & $1/1800$ \\ \hline
  initialization error & $0.1270$ & $0.1255$ & $0.1284$ & $0.1323$ \\ \hline
  recovery error & $0.0698$ & $0.0659$ & $0.0683$ & $0.0719$  \\ \hline
  dual gap & $7.8\times10^{-8}$ & $3.9\times10^{-5}$ & $0.1553$ & $0.5112$ \\ \hline
  $\lambda_{n-1}(\nabla_{\X}f(\X^*,\y^*))-\lambda_{n}(\nabla_{\X}f(\X^*,\y^*))$ & $39.8591$ & $78.9982$ & $150.3524$ & $217.06$  \\ \hline 
  $\Vert\diag(\X^*)-\overrightarrow{\mathbf{1}}\Vert_2$ & $3.2\times10^{-10}$ & $2.1\times10^{-8}$ & $5.1\times10^{-7}$ & $3.7\times10^{-7}$ \\ \hline
   \end{tabular}
  \caption{Numerical results for the phase synchronization problem.
  }\label{table:phaseSync}
\end{center}
}
\vskip -0.2in
\end{table*}\renewcommand{\arraystretch}{1}

\subsection{Linearly constrained low-rank matrix estimation}

Consider the following penalized formulation:
\begin{align*}
\min_{\substack{\trace(\X)=1,\\ \X\succeq 0}} \langle\X,-\M\rangle+\lambda\Vert\mathcal{A}(\X)-\b\Vert_2 = \min_{\substack{\trace(\X)=1,\\ \X\succeq 0}} \max_{\Vert\y\Vert_2\le1}\langle\X,-\M\rangle+\lambda\langle\mathcal{A}(\X)-\b,\y\rangle,
\end{align*}

where $\M = \z_0\z_0^{\top} + \frac{c}{2}(\N+\N^{\top})$ is the noisy observation of some rank-one matrix $\z_0\z_0^{\top}$ such that $\Vert\z_0\Vert_2=1$ and the noise matrix is chosen $\N\sim\mathcal{N}(0,\I_n)$. We take $\mathcal{A}(\X)=(\langle\A_1,\X\rangle,\ldots,\langle\A_m,\X\rangle)^{\top}$ with matrices $\A_1,\ldots,\A_m\in\mathbb{S}^n$ of the form $\A_i=\v_i\v_i^{\top}$ such that $\v_i\sim\mathcal{N}(0,1)$. We take $\b\in\reals^m$ such that $b_i=\langle\A_i,\z_0\z_0^{\top}\rangle$.

We initialize the $\X$ variable with the rank-one approximation of $\M$. That is, we take $\X_1=\u_1\u_1^{\top}$, where $\u_1$ is the top eigenvector of $\M$. The $\y$ variable is initialized with $\y_1=(\mathcal{A}(\X_1)-\b)/\Vert\mathcal{A}(\X_1)-\b\Vert_2$.

We set the number of constraints to $m=n$, the penalty parameter to $\lambda=2$, and the step-size to $\eta=1/(2\lambda)$. We set the number of iterations in each experiment to $T=2000$ and for each value of $n$ we average the measurements over  $10$ i.i.d. runs.

%\begin{table*}[!htb]\renewcommand{\arraystretch}{1.3}
%{\small
%\begin{center}
%  \begin{tabular}{| l | c | c | c | c | c | c | c | c | c |} \hline 
%  dimension (n) & $100$ & $200$ & $400$ & $600$ \\ \hline
%  noise level (c) & $0.0365$ & $0.0182$ & $0.0060$ & $0.0040$ \\ \hline
%$\lambda$ & $2$ & $2$ & $0.01$ & $0.01$ \\ \hline
%  initialization error & $0.1219$ & $0.0630$ & $0.1684$ & $0.1655$ \\ \hline
%  recovery error & $0.0437$ & $0.0234$ & $0.0140$ & $0.0092$  \\ \hline
%  dual gap & $5.3\times10^{-11}$ & $2.8\times10^{-8}$ & $1.7\times10^{-11}$ & $4.2\times10^{-7}$ \\ \hline
%  gap & $0.2941$ & $0.4374$ & $0.8384$ & $0.8675$  \\ \hline
%  $\Vert\mathcal{A}(\X^*)-\b\Vert_2$ & $0.0080$ & $0.0038$ \\ \hline
%   \end{tabular}
%  \caption{Numerical results for the linearly constrained low-rank matrix estimation problem. $\textsc{SNR}=0.15$.
%  }\label{table:robustPCArank10}
%\end{center}
%}
%\vskip -0.2in
%\end{table*}\renewcommand{\arraystretch}{1}

\begin{table*}[!htb]\renewcommand{\arraystretch}{1.3}
{\footnotesize
\begin{center}
  \begin{tabular}{| l | c | c | c | c | c | c | c | c | c |} \hline 
  dimension (n) & $100$ & $200$ & $400$ & $600$ \\ \hline
%  noise level (c) & $0.0365$ & $0.0257$ & $0.0177$ & $0.0143$ \\ \hline
  $\textrm{SNR}$ & $0.15$ & $0.075$ & $0.04$ & $0.027$ \\ \hline
% $\lambda$ & $2$ & $2$ & $2$ & $2$ \\ \hline
  initialization error & $0.1219$ & $0.1324$ & $0.1242$ & $0.1228$ \\ \hline
  recovery error & $0.0437$ & $0.0617$ & $0.0685$ & $0.0735$  \\ \hline
  dual gap & $5.3\times10^{-11}$ & $5.0\times10^{-12}$ & $8.5\times10^{-12}$ & $2.3\times10^{-11}$ \\ \hline
  $\lambda_{n-1}(\nabla_{\X}f(\X^*,\y^*))-\lambda_{n}(\nabla_{\X}f(\X^*,\y^*))$ & $0.2941$ & $0.3409$ & $0.4690$ & $0.5069$  \\ \hline
  $\Vert\mathcal{A}(\X^*)-\b\Vert_2$ & $0.0080$ & $0.0082$ & $0.0079$ & $0.0073$ \\ \hline
   \end{tabular}
  \caption{Numerical results for the linearly constrained low-rank matrix estimation problem. 
  }\label{table:linearlyConstrained}
\end{center}
}
\vskip -0.2in
\end{table*}\renewcommand{\arraystretch}{1}

\section{Discussion}
This work expands upon a line of research that aims to harness the ability of convex relaxations to produce low-rank and high-quality solutions to important low-rank matrix optimization problems, while insisting on methods that, at least locally, store and manipulate only low-rank matrices. Focusing on the challenging case of nonsmooth objective functions and following our evidence for the difficulties of obtaining such a result for subgradient methods  (\cref{lemma:negativeExampleNonSmooth}), we consider tackling nonsmooth objectives via saddle-point formulations. We  prove that indeed under a generalized strict complementarity condition, a state-of-the-art method for convex-concave saddle-point problems converges locally while storing and manipulating only low-rank matrices. Extensive experiments over several tasks demonstrate that our conceptual approach of utilizing low-rank projections for more efficient optimization is not only of theoretical merit, but indeed seems to work well in practice.

\section*{Acknowledgements}
This research was supported by the ISRAEL SCIENCE FOUNDATION (grant No. 1108/18).

\bibliographystyle{plain}
\bibliography{bibli}

\appendix
\appendixpage

%\section{Proofs omitted from Section \ref{sec:strictComp}} \label{appendix:1}

\section{Proof of Lemma \ref{lemma:kkt}}\label{sec:proofKKT}
We first restate the lemma and then prove it.

\begin{lemma} 
Let $\X^*\in\Sn$ be a rank-$r^*$ optimal solution to Problem \eqref{nonSmoothProblem}. $\X^*$ satisfies the (standard) strict complementarity assumption with parameter $\delta>0$ if and only if there exists a subgradient $\G^*\in\partial g(\X^*)$ such that $\langle \X-\X^*,\G^*\rangle\ge0$ for all $\X\in\mathcal{S}_n$ and $\lambda_{n-r^*}(\G^*)-\lambda_{n}(\G^*)\ge\delta$.
\end{lemma}

\begin{proof}
By Slater's condition strong duality holds for Problem \eqref{nonSmoothProblem}. Therefore, the KKT conditions for Problem \eqref{nonSmoothProblem} hold for the optimal solution $\X^*$ and some optimal dual solution $(\Z^*,s^*)$. 
The Lagrangian of Problem \eqref{nonSmoothProblem} can be written as
\begin{align*}
\mathcal{L}(\X,\Z,s)=g(\X)+s(1-\trace(\X))-\langle\Z,\X\rangle.
\end{align*}
Thus, using the generalized KKT conditions for nonsmooth optimization problems (see Theorem 6.1.1 in \cite{generalizedKKT}), this implies that for the primal and dual optimal solutions
\begin{align*}
\mathbf{0}\in \partial g(\X^*)-\Z^*-s^*\I,
\\ \langle \X^*,\Z^*\rangle=0,
\\ \trace(\X^*)=1,
\\ \X^*,\Z^*\succeq0.
\end{align*}
The generalized first order optimality condition for unconstrained minimization implies that there exists some $\G^*\in \partial g(\X^*)$ for which $\mathbf{0}= \G^*-\Z^*-s^*\I$. It remains to be shown that $\langle \X-\X^*,\G^*\rangle\ge0$ for all $\X\in\mathcal{S}_n$.

The cone of positive semidefinite matrices is self-dual, that is $\Z^*\succeq0$ if and only if $\langle \X,\Z^*\rangle\ge0$ for all $\X\in\mathcal{S}_n$.
Therefore, $\Z^*\succeq0$ if and only if for all $\X\in\mathcal{S}_n$ it holds that
\begin{align*}
0\le \langle \X,\Z^*\rangle & = \langle \X,\Z^*\rangle-\langle \X^*,\Z^*\rangle+s^*\langle \X-\X^*,\I\rangle
= \langle \X-\X^*,\Z^*+s^*\I\rangle
\\ & = \langle \X-\X^*,\G^*\rangle
\end{align*}
as desired. The first equality holds using the complementarity condition and the property that $\trace(\X)=\trace(\X^*)=1$.

Using the equality $\G^*=\Z^*+s^*\I$ it holds that 
\begin{align*}
\lambda_{n-r^*}(\Z^*) & = \lambda_{n-r^*}(\Z^*)+s^*\I-\lambda_{n}(\Z^*)-s^*\I = \lambda_{n-r^*}(\Z^*+s^*\I)-\lambda_{n}(\Z^*+s^*\I)
\\ & =\lambda_{n-r^*}(\G^*)-\lambda_{n}(\G^*).
\end{align*}
Thus, $\X^*$ satisfies the strict complementarity assumption with parameter $\delta>0$, i.e., $\lambda_{n-r^*}(\Z^*)\ge\delta$, if and only if $\lambda_{n-r^*}(\G^*)-\lambda_{n}(\G^*)\ge\delta$.
\end{proof}

%\section{Proofs omitted from Section \ref{sec:smooth2Saddle}}

\section{Proof of Lemma \ref{lemma:connectionSubgradientNonSmoothAndSaddlePoint}}\label{sec:proofSubgradEquiv}
We first restate the lemma and then prove it.
\begin{lemma} 
If $(\X^*,\y^*)$ is a saddle-point of Problem \eqref{problem1} then $\X^*$ is an optimal solution to Problem \eqref{nonSmoothProblem},  $\nabla_{\X}f(\X^*,\y^*)\in\partial g(\X^*)$,  and for all $\X\in\Sn$ it holds that $\langle\X-\X^*,\nabla_{\X}f(\X^*,\y^*)\rangle\ge0$.
Conversely, under \cref{ass:struct}, if $\X^*$ is an optimal solution to Problem \eqref{nonSmoothProblem}, and $\G^*\in\partial{}g(\X^*)$ which satisfies $\langle\X-\X^*,\G^*\rangle\ge0$  for all $\X\in\Sn$, then there exists $\y^*\in\argmax_{\y\in\mathcal{K}}f(\X^*,\y)$ such that $(\X^*,\y^*)$ is a saddle-point of Problem \eqref{problem1}, and $\nabla_{\X}f(\X^*,\y^*)=\G^*$.
\end{lemma}

\begin{proof}

%Comparing the two representations of the problem,
%\begin{align*}
%\max_{\y\in\mathcal{K}}f(\X,\y) = g(\X).
%\end{align*}
For the first direction of the lemma, we first observe that for any $\X_1,\X_2\in\Sn$ and $\widetilde{\y}_1\in \argmax_{\y\in\mathcal{K}}f(\X_1,\y)$, $\widetilde{\y}_2\in \argmax_{\y\in\mathcal{K}}f(\X_2,\y)$, using the gradient inequality for $f(\cdot,\widetilde{\y}_2)$, it holds that
\begin{align*}
g(\X_1) & = f(\X_1,\widetilde{\y}_1) \ge f(\X_1,\widetilde{\y}_2) \ge f(\X_2,\widetilde{\y}_2) + \langle\nabla_{\X}f(\X_2,\widetilde{\y}_2), \X_1-\X_2\rangle 
\\ & = g(\X_2)+\langle\nabla_{\X}f(\X_2,\widetilde{\y}_2), \X_1-\X_2\rangle.
\end{align*}
Thus,  $\nabla_{\X}f(\X_2,\widetilde{\y}_2)$ is a subgradient of $g(\cdot)$ at $\X_2$, i.e., $\nabla_{\X}f(\X_2,\widetilde{\y}_2)\in\partial g(\X_2)$.

In particular, for a saddle-point $(\X^*,\y^*)\in{\Sn\times\mathcal{K}}$ it holds that $\y^*\in \argmax_{\y\in\mathcal{K}}f(\X^*,\y)$, and therefore, it follows that $\nabla_{\X}f(\X^*,\y^*)\in\partial g(\X^*)$. In addition, for all $\X\in\Sn$ and $\widetilde{\y}\in \argmax_{\y\in\mathcal{K}}f(\X,\y)$ we have
\[ g(\X^*)=f(\X^*,\y^*)\le f(\X,\y^*)\le f(\X,\widetilde{\y})=g(\X),\] which implies that $\X^*$ is an optimal solution to $\min_{\X\in\Sn}g(\X)$. 

Finally, we need to show that the subgradient $\nabla_{\X}f(\X^*,\y^*)\in\partial g(\X^*)$ indeed satisfies the first-order optimality condition for $g(\cdot)$ at $\X^*$. To see this, we observe that since $\X^*$ is an optimal solution to $\min_{\X\in\Sn}f(\X,\y^*)$, it follows from the first-order optimality condition for the problem $\min_{\X\in\mS_n}f(\X,\y^*)$, that for all $\W\in\Sn$
\begin{align*}
\langle \W-\X^*,\nabla_{\X}f(\X^*,\y^*)\rangle\ge0,
\end{align*}
as needed.

For the second direction, let $\X^*\in\argmin_{\X\in\Sn} g(\X)$ and let $\G^*\in\partial g(\X^*)$ such that $\langle \X-\X^*,\G^*\rangle\ge0$ for all $\X\in\mS_n$. By Assumption \ref{ass:struct} and using Danskin's theorem (see for instance \cite{DanskinTheorem}), the subdifferential set of $g(\X^*) = h(\X^*) + \max_{\y\in\mK}\y^{\top}(\mA(\X^*)-\b)$ can be written as
\begin{align*}
\partial g(\X^*) & = \nabla{}h(\X^*)+\textrm{conv}\left\lbrace\mA^{\top}(\y)\ \Big\vert\ \y\in\argmax_{\y\in\mathcal{K}}\y^{\top}(\mA(\X^*)-\b)\right\rbrace
\\ & = \nabla{}h(\X^*)+\mA^{\top}\left(\textrm{conv}\left\lbrace\y\ \Big\vert\ \y\in\argmax_{\y\in\mathcal{K}}\y^{\top}(\mA(\X^*)-\b)\right\rbrace\right)
\\ & = \nabla{}h(\X^*)+\mA^{\top}\left(\left\lbrace\y\ \Big\vert\ \y\in\argmax_{\y\in\mathcal{K}}\y^{\top}(\mA(\X^*)-\b)\right\rbrace\right)
\\ & = \nabla{}h(\X^*)+\mA^{\top}\left(\left\lbrace\y\ \Big\vert\ \y\in\argmax_{\y\in\mathcal{K}}f(\X^*,\y)\right\rbrace\right),
\end{align*}
where $\textrm{conv}\lbrace\cdot\rbrace$ denotes the convex hull operation and the third equality follows from the convexity of $\mathcal{K}$.

Thus, there exists some $\y^*\in\argmax_{\y\in\mathcal{K}}f(\X^*,\y)$ such that $\G^*=\nabla{}h(\X^*)+\mA^{\top}(\y^*)=\nabla_{\X}f(\X^*,\y^*)$.

Since $\y^*\in \argmax_{\y\in\mathcal{K}}f(\X^*,\y)$, it follows that for all $\y\in\mathcal{K}$, $f(\X^*,\y^*)\ge f(\X^*,\y)$. In addition, using the fact that $\G^*$ satisfies the first-order optimality condition, and using gradient inequality w.r.t. $f(\cdot,\y^*)$, we have that for all $\X\in\Sn$,
\begin{align*}
0\le\langle \X-\X^*,\G^*\rangle=\langle \X-\X^*,\nabla_{\X}f(\X^*,\y^*)\rangle\le f(\X,\y^*)-f(\X^*,\y^*).
\end{align*}
Thus, it follows that $f(\X,\y^*)\ge f(\X^*,\y^*)$. Therefore, $(\X^*,\y^*)$ is indeed a saddle-point of $f$.

\end{proof}

\section{Proof of \cref{lemma:EGconvergenceRate}} \label{appendixConvergenceProof}
We first restate the lemma and then prove it.

\begin{lemma}
Let $\lbrace(\X_t,\y_t)\rbrace_{t\ge1}$ and $\lbrace(\Z_{t},\w_t)\rbrace_{t\ge2}$ be the sequences generated by \cref{alg:EG} with a fixed step-size $\eta_t=\eta\le\min\left\lbrace\frac{1}{\beta_{X}+\beta_{Xy}},\frac{1}{\beta_{y}+\beta_{yX}},\frac{1}{\beta_{X}+\beta_{yX}},\frac{1}{\beta_{y}+\beta_{Xy}}\right\rbrace$ then
%\textcolor{red}{
\begin{align*}
\max_{\y\in\mathcal{K}} f\left(\frac{1}{T}\sum_{t=1}^T \Z_{t+1},\y\right) -  \min_{\X\in\Sn} f\left(\X,\frac{1}{T}\sum_{t=1}^T\w_{t+1}\right) & \le \frac{D^2}{2\eta T},
\end{align*}%}
where $D:=\sup_{(\X,\y),(\tilde{\X},\tilde{\y})\in{\Sn\times\mathcal{K}}}\Vert(\X,\y)-(\tilde{\X},\tilde{\y})\Vert$.
%In particular,
%\begin{enumerate}
%\item $\min_{t\in[T]}\left(\max_{\y\in\mathcal{K}} f(\Z_{t+1},\y) - \min_{\X\in\Sn} f(\X,\w_{t+1})\right)
%\le \frac{D^2}{2\eta T}$.
%\item $\max_{\y\in\mathcal{K}} f\left(\frac{1}{T}\sum_{t=1}^T \Z_{t+1},\y\right) -  \min_{\X\in\Sn} f\left(\X,\frac{1}{T}\sum_{t=1}^T\w_{t+1}\right) \le \frac{D^2}{2\eta T}$.
%\end{enumerate}
\end{lemma}

\begin{proof}

The projection theorem states that projecting some point $\s$ onto some closed and convex set $\mathcal{C}$ satisfies that for all $\z\in\mathcal{C}$ it holds that $\langle\Pi_{\mathcal{C}}[\s]-\s,\Pi_{\mathcal{C}}[\s]-\z\rangle\le0$. In particular, for any $\X\in\Sn$, using the updates for $\X_{t+1}$ and $\Z_{t+1}$, the two following inequalities hold:
\begin{align}
\eta_t\langle\Z_{t+1}-\X,\nabla_{\X}f(\X_t,\y_t)\rangle & \le \langle\X_t-\Z_{t+1},\Z_{t+1}-\X\rangle \label{ineq:projZ}
\\ \eta_t\langle\X_{t+1}-\X,\nabla_{\X}f(\Z_{t+1},\w_{t+1})\rangle & \le \langle\X_t-\X_{t+1},\X_{t+1}-\X\rangle. \label{ineq:projX}
\end{align}

By the gradient inequality, for any $\X\in\Sn$
\begin{align} \label{ineq:EGconvergenceX}
& f(\Z_{t+1},\w_{t+1})-f(\X,\w_{t+1}) \nonumber \\ & \le \langle \Z_{t+1}-\X,\nabla_{\X}f(\Z_{t+1},\w_{t+1})\rangle \nonumber
\\ & = \langle\X_{t+1}-\X,\nabla_{\X}f(\Z_{t+1},\w_{t+1})\rangle + \langle\Z_{t+1}-\X_{t+1},\nabla_{\X}f(\X_{t},\y_{t})\rangle \nonumber
\\ & \ \ \ + \langle\Z_{t+1}-\X_{t+1},\nabla_{\X}f(\Z_{t+1},\w_{t+1})-\nabla_{\X}f(\X_{t},\y_{t}) \rangle.
\end{align}
We will bound these three terms separately.

For the first term, using \eqref{ineq:projX} and the Pythagoras identity
\begin{align} \label{ineq:EGconvergenceTerm1}
& \langle\X_{t+1}-\X,\nabla_{\X}f(\Z_{t+1},\w_{t+1})\rangle \nonumber
\\ & \le \frac{1}{\eta_t} \langle\X_t-\X_{t+1},\X_{t+1}-\X\rangle \nonumber
\\ & = -\frac{1}{2\eta_t}\Vert\X_t-\X_{t+1}\Vert_F^2 + \frac{1}{2\eta_t}\Vert\X_t-\X\Vert_F^2 - \frac{1}{2\eta_t}\Vert\X_{t+1}-\X\Vert_F^2.
\end{align}

For the second term, using \eqref{ineq:projZ} with $\X=\X_{t+1}$ and the Pythagoras identity
\begin{align} \label{ineq:EGconvergenceTerm2}
& \langle\Z_{t+1}-\X_{t+1},\nabla_{\X}f(\X_{t},\y_{t})\rangle \nonumber
\\ & \le \frac{1}{\eta_t}\langle\X_t-\Z_{t+1},\Z_{t+1}-\X_{t+1}\rangle \nonumber
\\ & = -\frac{1}{2\eta_t}\Vert\X_t-\Z_{t+1}\Vert_F^2 + \frac{1}{2\eta_t}\Vert\X_t-\X_{t+1}\Vert_F^2 - \frac{1}{2\eta_t}\Vert\Z_{t+1}-\X_{t+1} \Vert_F^2.
\end{align}

For the third term, using the Cauchy–Schwarz inequality, the $\beta_{X}$ and $\beta_{Xy}$ smoothness, and the inequality $2ab\le a^2+b^2$ we obtain
\begin{align} \label{ineq:EGconvergenceTerm3}
& \langle\Z_{t+1}-\X_{t+1},\nabla_{\X}f(\Z_{t+1},\w_{t+1})-\nabla_{\X}f(\X_{t},\y_{t}) \rangle \nonumber
\\ & \le \Vert\Z_{t+1}-\X_{t+1}\Vert_F\cdot\Vert\nabla_{\X}f(\Z_{t+1},\w_{t+1})-\nabla_{\X}f(\X_{t},\y_{t})\Vert_F \nonumber
\\ & \le \Vert\Z_{t+1}-\X_{t+1}\Vert_F\cdot\Vert\nabla_{\X}f(\Z_{t+1},\w_{t+1})-\nabla_{\X}f(\X_{t},\w_{t+1})\Vert_F \nonumber
\\ & + \Vert\Z_{t+1}-\X_{t+1}\Vert_F\cdot\Vert\nabla_{\X}f(\X_{t},\w_{t+1})-\nabla_{\X}f(\X_{t},\y_{t})\Vert_F \nonumber
\\ & \le \left(\beta_{X}\Vert\Z_{t+1}-\X_{t}\Vert_F+\beta_{Xy}\Vert\w_{t+1}-\y_{t}\Vert_2\right)\cdot\Vert\Z_{t+1}-\X_{t+1}\Vert_F \nonumber
\\ & \le \frac{\beta_{X}}{2}\left(\Vert\Z_{t+1}-\X_{t}\Vert_F^2+\Vert\Z_{t+1}-\X_{t+1}\Vert_F^2\right)+\frac{\beta_{Xy}}{2}\left(\Vert\w_{t+1}-\y_{t}\Vert_2^2+\Vert\Z_{t+1}-\X_{t+1}\Vert_F^2\right).
\end{align}

Plugging \eqref{ineq:EGconvergenceTerm1}, \eqref{ineq:EGconvergenceTerm2}, and \eqref{ineq:EGconvergenceTerm3} into \eqref{ineq:EGconvergenceX} we obtain 
\begin{align*}
& f(\Z_{t+1},\w_{t+1})-f(\X,\w_{t+1}) 
\\ & \le \frac{1}{2\eta_t}\left(\Vert\X_t-\X\Vert_F^2-\Vert\X_{t+1}-\X\Vert_F^2\right) +\left(\frac{\beta_{X}}{2}-\frac{1}{2\eta_t}\right)\Vert\Z_{t+1}-\X_{t}\Vert_F^2
\\ & \ \ \ +\left(\frac{\beta_{X}+\beta_{Xy}}{2}-\frac{1}{2\eta_t}\right)\Vert\Z_{t+1}-\X_{t+1}\Vert_F^2+\frac{\beta_{Xy}}{2}\Vert\w_{t+1}-\y_{t}\Vert_2^2.
\end{align*}

Using similar arguments, for any $\y\in\mathcal{K}$ 
\begin{align*}
& f(\Z_{t+1},\y) - f(\Z_{t+1},\w_{t+1}) 
\\ & \le \frac{1}{2\eta_t}\left(\Vert\y_t-\y\Vert_2^2-\Vert\y_{t+1}-\y\Vert_2^2\right) +\left(\frac{\beta_{y}}{2}-\frac{1}{2\eta_t}\right)\Vert\w_{t+1}-\y_{t}\Vert_2^2
\\ & \ \ \ +\left(\frac{\beta_{y}+\beta_{yX}}{2}-\frac{1}{2\eta_t}\right)\Vert\w_{t+1}-\y_{t+1}\Vert_2^2+\frac{\beta_{yX}}{2}\Vert\Z_{t+1}-\X_{t}\Vert_F^2.
\end{align*}

Summing the last two inequalities, we obtain for $\eta_t\le\min\left\lbrace\frac{1}{\beta_{X}+\beta_{Xy}},\frac{1}{\beta_{y}+\beta_{yX}},\frac{1}{\beta_{X}+\beta_{yX}},\frac{1}{\beta_{y}+\beta_{Xy}}\right\rbrace$
\begin{align*}
f(\Z_{t+1},\y) - f(\X,\w_{t+1}) & \le \frac{1}{2\eta_t}\left(\Vert(\X_t,\y_t)-(\X,\y)\Vert^2-\Vert(\X_{t+1},\y_{t+1})-(\X,\y)\Vert^2\right)
\\ & \ \ \ +\left(\frac{\beta_{X}+\beta_{yX}}{2}-\frac{1}{2\eta_t}\right)\Vert\Z_{t+1}-\X_{t}\Vert_F^2 \\
& \ \ \ +\left(\frac{\beta_{y}+\beta_{Xy}}{2}-\frac{1}{2\eta_t}\right)\Vert\w_{t+1}-\y_{t}\Vert_2^2
\\ & \ \ \ +\left(\frac{\beta_{X}+\beta_{Xy}}{2}-\frac{1}{2\eta_t}\right)\Vert\Z_{t+1}-\X_{t+1}\Vert_F^2\\
& \ \ \ +\left(\frac{\beta_{y}+\beta_{yX}}{2}-\frac{1}{2\eta_t}\right)\Vert\w_{t+1}-\y_{t+1}\Vert_2^2
\\ & \le \frac{1}{2\eta_t}\left(\Vert(\X_t,\y_t)-(\X,\y)\Vert^2-\Vert(\X_{t+1},\y_{t+1})-(\X,\y)\Vert^2\right).
\end{align*}

%\textcolor{red}{
Averaging over $t=1,\ldots,T$ and taking a $\eta_t=\eta$
\begin{align*}
\frac{1}{T}\sum_{t=1}^T \left( f(\Z_{t+1},\y) -  f(\X,\w_{t+1})\right) 
 \le \frac{1}{2\eta T}\max_{(\X,\y)\in{\Sn\times\mathcal{K}}}\Vert(\X_1,\y_1)-(\X,\y)\Vert^2
 \le \frac{D^2}{2\eta T}.
\end{align*}
Taking the maximum over all $\y\in\mathcal{K}$ and minimum over all $\X\in\Sn$
and using the convexity of $f(\cdot,\y)$ and concavity of $f(\X,\cdot)$, 
\begin{align*}
& \max_{\y\in\mathcal{K}} f\left(\frac{1}{T}\sum_{t=1}^T \Z_{t+1},\y\right) -  \min_{\X\in\Sn} f\left(\X,\frac{1}{T}\sum_{t=1}^T\w_{t+1}\right) 
\\ & \le \max_{\y\in\mathcal{K}} \frac{1}{T}\sum_{t=1}^T f\left( \Z_{t+1},\y\right) -  \min_{\X\in\Sn} \frac{1}{T}\sum_{t=1}^T f\left(\X,\w_{t+1}\right)
 \le \frac{D^2}{2\eta T}.
\end{align*}
%}

%Therefore, 
%\begin{align*}
%\min_{t\in[T]} \max_{\y\in\mathcal{K}} f(\Z_{t+1},\y) - \max_{t\in[T]} \min_{\X\in\Sn}f(\X,\w_{t+1}) & \le \frac{D^2}{2\eta T}.
%\end{align*}

%Averaging over $t=1,\ldots,T$ and taking a $\eta_t=\eta$
%\begin{align*}
%\min_{t\in[T]} f(\Z_{t+1},\y) - \max_{t\in[T]} f(\X,\w_{t+1}) & \le \frac{1}{T}\sum_{t=1}^T f(\Z_{t+1},\y) - \frac{1}{T}\sum_{t=1}^T f(\X,\w_{t+1}) 
%\\ & \le \frac{1}{2\eta T}\Vert(\X_1,\y_1)-(\X,\y)\Vert^2.
%\end{align*}
%
%Therefore, 
%\begin{align*}
%\min_{t\in[T]} \max_{\y\in\mathcal{K}} f(\Z_{t+1},\y) - \max_{t\in[T]} \min_{\X\in\Sn}f(\X,\w_{t+1}) & \le \frac{D^2}{2\eta T}.
%\end{align*}

\end{proof}

\section{Calculating the dual-gap in saddle-point problems} \label{appendixDualGapCalculation}

Set some point $(\widehat{\Z},\widehat{\w})\in{\lbrace\trace(\X)=\tau,\ \X\succeq 0\rbrace\times\mathcal{K}}$. Using the concavity of $f(\widehat{\Z},\cdot)$ and convexity of $f(\cdot,\widehat{\w})$, for all $\y\in\mathcal{K}$ and $\X\in\lbrace\trace(\X)=\tau,\ \X\succeq 0\rbrace$, it holds that
\begin{align*}
f(\widehat{\Z},\y)-f(\widehat{\Z},\widehat{\w}) & \le \langle\widehat{\w}-\y,-\nabla_{\y}f(\widehat{\Z},\widehat{\w})\rangle, 
\\ f(\widehat{\Z},\widehat{\w})-f(\X,\widehat{\w}) & \le \langle\widehat{\Z}-\X,\nabla_{\X}f(\widehat{\Z},\widehat{\w})\rangle.
\end{align*}

By taking the maximum of all $\y\in\mathcal{K}$ we obtain in particular that
\begin{align*}
f(\X^*,\y^*)-f(\widehat{\Z},\widehat{\w}) &  \le f(\widehat{\Z},\y^*)-f(\widehat{\Z},\widehat{\w}) \le \max_{\y\in\mathcal{K}}f(\widehat{\Z},\y)-f(\widehat{\Z},\widehat{\w})
\\&  \le \max_{\y\in\mathcal{K}}\langle\widehat{\w}-\y,-\nabla_{\y}f(\widehat{\Z},\widehat{\w})\rangle,
\end{align*}
and taking the maximum of all $\X\in\lbrace\trace(\X)=\tau,\ \X\succeq 0\rbrace$
\begin{align*}
f(\widehat{\Z},\widehat{\w})-f(\X^*,\y^*) & \le f(\widehat{\Z},\widehat{\w})-f(\X^*,\widehat{\w}) \le f(\widehat{\Z},\widehat{\w})-\min_{\substack{\trace(\X)=\tau,\\ \X\succeq 0}}f(\X,\widehat{\w}) 
\\ & \le \max_{\substack{\trace(\X)=\tau,\\ \X\succeq 0}}\langle\widehat{\Z}-\X,\nabla_{\X}f(\widehat{\Z},\widehat{\w})\rangle.
\end{align*}

Summing these two inequalities, we obtain a bound on the dual-gap at $(\widehat{\Z},\widehat{\w})$ which can be written as
\begin{align*}
g(\widehat{\Z})-g^* & \le \max_{\y\in\mathcal{K}}f(\widehat{\Z},\y)-\min_{\X\in\Sn}f(\X,\widehat{\w}) \nonumber
\\ & \le \max_{\substack{\trace(\X)=\tau,\\ \X\succeq 0}}\langle\widehat{\Z}-\X,\nabla_{\X}f(\widehat{\Z},\widehat{\w})\rangle - \min_{\y\in\mathcal{K}} \langle\widehat{\w}-\y,\nabla_{\y}f(\widehat{\Z},\widehat{\w})\rangle.
\end{align*} 

It is easy to see that the maximizer of the first term in the RHS of the above is $\tau\v_n\v_n^{\top}$ where $\v_n$ is the smallest eigenvector of $\nabla_{\X}f(\widehat{\Z},\widehat{\W})$, and the minimizer of the second term is $\Y_{i,j}=\sign(\nabla_{\Y}f(\widehat{\Z},\widehat{\W})_{i,j})$ for $\mathcal{K}=\lbrace\Y\in\reals^{n\times n}\ \vert\ \Vert\Y\Vert_{\infty}\le1\rbrace$ and $\nabla_{\y}f(\widehat{\Z},\widehat{\w})/\Vert\nabla_{\y}f(\widehat{\Z},\widehat{\w})\Vert_2$ for $\mathcal{K}=\lbrace\y\in\reals^{n}\ \vert\ \Vert\y\Vert_{2}\le1\rbrace$.

\end{document}